\documentclass[a4paper, 10.5pt]{article}
\usepackage{amsmath}
\usepackage{amssymb,esint}
\usepackage{amscd}
\usepackage{xspace}
\usepackage{fancyhdr}
\usepackage{color}
\usepackage{srcltx} %jump between tex and dvi, used for kile
\newtheorem{theorem}{Theorem}[section]

\newtheorem{lemma}[theorem]{Lemma}

\newtheorem{proposition}[theorem]{Proposition}
\newtheorem{remark}[theorem]{Remark}

\newenvironment{proof}[1][Proof]{\textbf{#1.} }{\hfill\rule{0.5em}{0.5em}}
{\catcode`\@=11\global\let\AddToReset=\@addtoreset
\AddToReset{equation}{section}

\AddToReset{theorem}{section}

\begin{document}
\title{Global  Calder\'{o}n--Zygmund theory for parabolic $p$-Laplacian system: the case $1<p\leq \frac{2n}{n+2}$ }
\author{
	{\bf Ke Chen,\thanks{E-mail address: kchen18@fudan.edu.cn, Fudan University, 220 Handan Road, Yangpu, Shanghai, 200433, China.}~~Quoc-Hung Nguyen\thanks{E-mail address: qhnguyen@amss.ac.cn, Academy of Mathematics and Systems Science, Chinese Academy of Sciences,
			Beijing 100190, PR China.}~~and~Na Zhao\thanks{E-mail address: zhaona@shufe.edu.cn, School of Mathematics, Shanghai University of Finance and Economics, 777 Guoding Road, Yangpu, Shanghai, 200433, China. }}
	}
\maketitle
\begin{abstract}
	The aim of this paper is to establish global Calder\'{o}n--Zygmund theory 
	to parabolic $p$-Laplacian system: 
	$$
	u_t -\operatorname{div}(|\nabla u|^{p-2}\nabla u) = \operatorname{div} (|F|^{p-2}F)~\text{in}~\Omega\times (0,T)\subset \mathbb{R}^{n+1},
	$$
	proving that
	$$
	F\in L^q\Rightarrow \nabla u\in L^q,
	$$
	for any $q>\max\{p,\frac{n(2-p)}{2}\}$ and $p>1$. Acerbi
	and Mingione \cite{Acerbi07} proved this estimate in the case $p>\frac{2n}{n+2}$.  In this article we settle the case $1<p\leq \frac{2n}{n+2}$. We also treat systems with discontinuous  coefficients having small BMO (bounded mean
	oscillation) norm. 
\end{abstract}
% \tableofcontents
  \section{Introduction and main results}

In this paper, we consider the following inhomogeneous parabolic system:
 \begin{equation}\label{e1}
 	\left\{
 	\begin{array}
 		[c]{ll}%
 		u_t -\operatorname{div}(a(x,t)|\nabla u|^{p-2}\nabla u) = \operatorname{div} (|F|^{p-2}F)&\text{in }\Omega_T,\\
 		u=0&\text{on}\
 		\partial_{p} \Omega_T.
 		\\
 	\end{array}
 	\right.
 \end{equation}
 with a given function $F\in L^p(\Omega_T,\mathbb{R}^{N\times n})$. Here $a(x,t)$ satisfies
 \begin{equation}\label{cona}
 	c_a\leq  a(x,t)\leq C_a, \quad\quad\forall (x,t)\in \Omega_T
 \end{equation}
 for some constants $c_a,C_a>0$. And
 $\Omega_T:=\Omega\times (0,T), T>0$, is a cylindrical domain with parabolic boundary $\partial_p\Omega_T:=(\partial\Omega\times [0,T])\cup (\bar \Omega\times \{0\})$, where $\Omega\subset\mathbb{R}^n$, $n\geq 3$, is a bounded open domain with $C^2$ boundary.
 
 Recall that $u\in C^0(0,T;L^2(\Omega,\mathbb{R}^N))\cap L^p(0,T;W^{1,p}(\Omega,\mathbb{R}^N))$ with $N\geq 1$ is a weak solution to \eqref{e1} if it satisfies
 \begin{align*}
 	\int_{\Omega_T}(u\phi_t-a(x,t)|\nabla u|^{p-2}\nabla u\cdot \nabla \phi-|F|^{p-2}F\cdot \nabla \phi) dxdt=0
 \end{align*}
 for all $\phi\in C_0^1(\Omega_T)$. The  Calder\'{o}n--Zygmund estimate is a classical and interesting problem in partial differential equations. The aim of this type of estimates is to explore whether the solution $u$ is as good integrable as the inhomogeneity $F$, i.e.,
 \begin{align}\label{cz}
 	F\in L^q(\Omega_T)\Rightarrow \nabla u\in L^q(\Omega_T)
 \end{align}
 holds for any $q\geq p$ or not. It is well-known that if $|F|\in L^p(\Omega_T)$, then the equation \eqref{e1} has a unique weak solution $u$ satisfying the following Calder\'{o}n--Zygmund estimate 
 \begin{align}\label{DuLp}
 	\sup_{t\in[0,T]}\|u(\cdot,t)\|_{L^2(\Omega)}+\|\nabla u\|_{L^p(\Omega_T)}\leq C\| F\|_{L^p(\Omega_T)}.
 \end{align}
 It is natural to ask whether we can extend the result to various function spaces. There are many significant results on this topic, including both elliptic and parabolic equations. See for example \cite{Acer05,Acerbi07,Boge,Peral,Iwa,KinZ01,Ming07,Ming10,Hung2,QH1,ThanhHung} and references therein. In this paper, we will present Calder\'{o}n--Zygmund type estimates for weak solutions $u$ to \eqref{e1} in the case $$1<p\leq \frac{2n}{n+2}.$$ Our method can also be applied to obtain Calder\'{o}n--Zygmund estimates for all  $p>\frac{2n}{n+2}$ and $q\geq p$ (see Remark \ref{repl0}).
 
 We first mention some previous results in related directions.
 For the elliptic case with distributional data, we have the full result, i.e. \eqref{cz} holds for any $q\geq p$, $p>1$. The result was essentially obtained by Iwaniec \cite{Iwa} in the scalar
 case ($N=1$) and by DiBenedetto and Manfredi \cite{DiBeM93} in the vectorial case ($N>1$). Kinnunen and Zhou \cite{KinZ99} extended this to elliptic equations with VMO coefficients. Based on the techniques in \cite{Peral}, a general Calderón–Zygmund theory for elliptic equations which admit non-standard growth conditions was obtained by Acerbi
 and Mingione in \cite{Acer05}. For global results with homogeneous
 Dirichlet boundary data, we refer to \cite{KinZ01} for   $C^{1,\alpha}$ domains, and \cite{Byun08} for Reifenberg domains. 
 
 We also mention some significant results concerning elliptic $p$-Laplacian type equations with measure data. 
 Mingione \cite{Ming07} established a comparison estimate and applied it for $p\geq 2$ to obtain the regularity for gradients of solutions. Then in \cite{Ming10}, Mingione introduced the idea of using fractional maximal functions of order 1 to obtain a certain level set decay lemma that involves the Hardy-Littlewood maximal function of the gradient and the fractional maximal function of the datum. This lemma yields a local $L^q$ gradient bound in the case $p\geq 2$. The comparison estimates established in \cite{Duzamin2} extended the result to the case $2-\frac{1}{n}<p<2$. In all the above works, the comparison estimates can not be applied to the range $1<p\leq 2-\frac{1}{n}$, in which case the distributional solutions 
 may not even belong to $W^{1,1}_{loc}$. 
 More recently, based on various
 tools developed for quasilinear equations with measure data and linear or nonlinear
 potential and Calder\'{o}n--Zygmund theories, (see \cite{Byun11,Duzamin1,Ming07,Ming10,55Ph2}), Phuc and the second author \cite{QH4} (see also \cite{HungPhuc3,HungPhuc4}) dealt with the singular case $\frac{3n-2}{2n-1}<p\leq 2-\frac{1}{n}$ for the first time. The key ingredients are some new local comparison estimates. Later they extended the result to the strongly singular case $1<p\leq\frac{3n-2}{2n-1}$ in \cite{QHPhuc}, eventually fulfilled the range $p>1$. 
 
 The problem in parabolic setting is more complicated. The main obstruction is the non-homogeneous scaling behavior of the problem with respect
 to space and time in the case $p\neq 2$ (then the equation is nonlinear), in the sense that the solution multiplied by a constant
 is in general not anymore a solution, even when $F=0$. The first result in this direction was obtained
 by Misawa \cite{Misa05} assuming that $F\in L^\infty$. In \cite{Acerbi07}, Acerbi and Mingione considered the equation \eqref{e1} with $\frac{2n}{n+2}<p<\infty$, and coefficient $a(x,t)\in VMO$ satisfying \eqref{cona}. They proved that if $|F|^p\in L^q(\Omega_T) $ for $q>1$, then $|\nabla u|^p\in L^q_{loc}(\Omega_T)$. The main idea to overcome the difficulties was to
 use the intrinsic geometry  for the upper-level sets of $|\nabla u|$ (which was introduced by DiBenedetto and Friedman \cite{ DiBenedetto1,DiBenedetto2,DiBenedetto3}) together with
 a maximal function free approach (see also \cite{Acer05}). Duzaar et al.\cite{Duzaar11} employed and partially modified the ideas to obtain the result for more general nonlinear systems. In \cite{Byun11}, the global Calder\'{o}n--Zygmund theory was investigated for a more general system. Considering the nonlinearity to have a small $BMO$ norm, and $\Omega$ is flat in
 Reifenberg’s sense, they obtain that  if $|F|^p\in L^q(\Omega_T) $ for $q>1$, then $|\nabla u|^p\in L^q(\Omega_T)$. Authors in \cite{Bui17,Bui17b} extended the results to more general setting of Lorentz spaces.  
 We also infer interested readers to \cite{55DuzaMing,KM,KM2,KM3,KM4,KM5,QH1} for potential estimates for the spatial gradient
 of solutions. 
 
  However, to the best of our knowledge, there are no such Calder\'{o}n--Zygmund estimates in the case $1<p\leq\frac{2n}{n+2}$. The main strategy to prove such estimates is combining higher regularity of the solution to the homogeneous system ($F=0$) with proper comparison estimates.  The range $p>\frac{2n}{n+2}$ is always unavoidable because of the structure of the homogeneous system. Take the classical $p$-Laplacian ($a(x,t)=1$) as an example. DiBenedetto and Friedman proved in \cite{DiBenedetto2} that for any weak solution $w $ to the homogeneous system, there holds
 $$
 \nabla w \in L^p \Rightarrow \nabla w \in L^\infty, \quad \text{if}\quad p>\frac{2n}{n+2}.
 $$
 They also proved the H\"{o}lder regularity for $\nabla w$. As shown by counterexamples in \cite{DiBenedetto1}, the above is not true for $1<p\leq \frac{2n}{n+2}$. 
 Choe  \cite{Choe1} extended the result in \cite{DiBenedetto2}, he proved that $\nabla w\in L^\infty$ provided
 $w\in L^r_{loc}$ for some $r>\frac{n(2-p)}{p}$. When $p> \frac{2n}{n+2}$, the integrability  $u\in L^r_{loc}$ is implicit in the notion of weak solution. While for $p\leq \frac{2n}{n+2}$, one needs to impose the integrability condition to obtain the boundedness of the solution. The integrability condition (as well as the range $p>\frac{2n}{n+2}$) is sharp as shown by counterexamples in \cite{DiBenedetto1} (see also \cite{Choe1,Choe2,DiBenedetto2,DiBenedetto3,Duzamin2} for more discussion). 
 
 In this paper, we will consider system \eqref{e1} with $1<p\leq\frac{2n}{n+2}$ and prove that  for any $q>\frac{n(2-p)}{2}$, $\nabla u\in L^q(\Omega_T)$ if $F\in L^q(\Omega_T)$ and $[a]_{R_0}\leq \delta_0$ for some $\delta_0, R_0>0$. Here we define the norm 
 \begin{equation*}
 	[a]_{r_0}:=\sup_{t\in [0,T]}\sup_{x\in \Omega, 0<r\leq r_0}\fint_{B_r(x)}\left| a(y,t)-(a(t))_{B_{r}(x)}\right|dx,
 \end{equation*}
 where $(a(t))_{B_{r}(x)}=\frac{1}{|B_{r}(x)|}\int_{B_{r}(x)} a(y,t) dy.$
 Note that $\frac{n(2-p)}{2}\geq p$ in our case. The key point is that we establish new comparison estimates to improve the integrability of the solution to the homogeneous system (see Lemma \ref{lemw}).
 
 We introduce some notations that will be used throughout the paper. We use the notation $X\lesssim Y$, which means that there exists an absolute constant $C>0$ such that $X\leq CY$. For any $\lambda, R>0$, we define the cylinder $Q_{R}^\lambda(x,t)=B_R(x)\times (-\lambda^{p-2}R^2+t,t)$ and extend $u=0$ for any $t<0$.  For simplicity in later statement, we introduce the function 
 \begin{equation}\label{formularG}
 	G_{\gamma}(s)=s^{\gamma}+s^{\frac{1}{\gamma}},\quad\quad \gamma\in(0,1).
 \end{equation}
 We are now ready to state the main result in this paper.
 \begin{theorem}\label{thm0}
 	Let $u$ be a weak solution of \eqref{e1}.   Then there exists $R_0>0$ small depending on $\partial\Omega$ such that for any $t_0\in[0,T]$, and  any $x_0,R$ satisfying
 	$$B_{4R}(x_0)\subset \Omega, \quad\quad\text{or}\quad\quad x_0\in \partial \Omega\  \text{and}\  0<R\leq  R_0,$$
 	there exists $\mathbf{h}\in (L^{\infty}(Q_{R/2}^\lambda(x_0,t_0)))^{N\times n}$ such that  for any $0<\eta,\kappa_0<1$ and $\frac{n(2-p)}{2}<\sigma_1<n+p$, there holds
 	\begin{align*}
 		\lambda\|\mathbf{h}\|_{L^\infty(Q_{R/2}^\lambda(x_0,t_0))}\leq&
 		C G_{\gamma_0}\left(\left(\fint_{Q_{2R}^\lambda(x_0,t_0)}|\lambda\nabla u|^{\sigma_1} \right)^{\frac{1}{\sigma_1}}\right)+ C(\eta,\kappa_0)G_{\gamma_0}\left(\left(\fint_{Q_{2R}^\lambda(x_0,t_0)}|\lambda F|^{\sigma_1} \right)^{\frac{1}{\sigma_1}}\right)\\
 		&\quad\quad\quad+\eta G_{\gamma_0}\left(\left(\fint_{Q_{2R}^\lambda(x_0,t_0)}|\lambda \nabla u|^{\sigma_1+\kappa_0}\right)^\frac{1}{\sigma_1+\kappa_0}\right)\nonumber,
 	\end{align*}
 	and
 	\begin{align*}
 		&\left(\fint_{Q_{R/2}^\lambda(x_0,t_0)}|\lambda(\nabla u-\mathbf{h})|^{\sigma_1} \right)^{\frac{1}{\sigma_1}}\\
 		&\quad\quad\leq \eta G_{\gamma_0}\left(\left(\fint_{Q_{2R}^\lambda(x_0,t_0)}|\lambda\nabla u|^{\sigma_1+\kappa_0} \right)^{\frac{1}{\sigma_1+\kappa_0}}\right) + C(\eta,\kappa_0)G_{\gamma_0}\left(\left(\fint_{Q_{2R}^\lambda(x_0,t_0)}|\lambda F|^{\sigma_1} \right)^{\frac{1}{\sigma_1}}\right)
 	\end{align*}
 	for some $\gamma_0\in (0,1)$, provided $\nabla u\in L^{\sigma_1+\kappa_0}(Q_{2R}^\lambda(x_0,t_0))$, $F\in L^{\sigma_1}(Q_{2R}^\lambda(x_0,t_0))$ and $[a]_{2R}\leq \eta^{\frac{n^4}{\kappa_0^4(p-1)^4}}$. 
 \end{theorem}
\begin{remark}
	Note that for $1<p\leq \frac{2n}{n+2}$, we can not replace the $L^{\sigma_1}$ norm of $\nabla u$ and $F$ in the right hand side by the $L^p$ norm, which finally leads to the range $q>\frac{n(2-p)}{2}$ in our final result.
\end{remark}
\begin{remark}\label{reth0}
	 If $p>\frac{2n}{n+2}$, then $\frac{n(2-p)}{2}<p$. Theorem \ref{thm0} holds for $\sigma_1=p$, $\kappa_0=0$ (see Remark \ref{repr1}).
\end{remark}
 Theorem \ref{thm0} implies $L^q$ bound for gradient of solution $u$. 
 \begin{theorem}\label{thm1} Let $q>\frac{n(2-p)}{2}$.
 	Let $u$ be a weak solution of \eqref{e1} with $F\in L^{q}(\Omega_T)$. We can find $\delta_0=\delta_0(n,p,q,\partial\Omega)\in(0,1)$ such that if $[a]_{R_0}\leq \delta_0$ for some $R_0>0$, we have the following estimate:
 	\begin{align*}
 		\int_{\Omega_T}|\nabla u|^{q} \leq C\int_{\Omega_T}|F|^{q}+C\left(\int_{\Omega_T}|F|^p\right)^{\frac{2-p+q}{2}},
 	\end{align*}
 	where $C=C(n,p,q,\Omega,T)\geq 1$ is a constant.
 \end{theorem}
 \begin{remark}
 	Note that for $q>\frac{n(2-p)}{2}$, there holds $	\frac{p(2-p+q)}{2}<q$. The above is the first  homogeneous result, in the sense that  $u=0$ if the data $F=0$. 
 \end{remark}
\begin{remark}\label{repl0}
	If $p>\frac{2n}{n+2}$, then Theorem \ref{thm1} holds for any $q\geq p$ (see Remark \ref{repl}).
\end{remark}
 Now, we give the sketch of proof. We first obtain a new comparison estimate on intrinsic parabolic cylinders $Q_{2R}^\lambda$, see Lemma \ref{u-w} and Lemma \ref{u-wbdry} for interior and boundary cases, respectively. These estimates work well in the case $1<p\leq \frac{2n}{n+2}$ and help us to improve the regularity of the solution to the homogeneous system. Moreover, the exponent of the norm of $F$ and $\nabla u$ can be close to $\frac{n(2-p)}{2}$ (see Remark \ref{mixcomparison}), which helps us to get sharp estimates in Proposition \ref{prop1} and Proposition \ref{prop22}. After that, we define the intrinsic $\lambda$-maximal function $\mathbf{M}^\lambda f$ in \eqref{Mlamb}. With new comparison estimates in hand, using the idea in \cite{QH4,QHPhuc}, we apply Theorem \ref{thm0} to obtain the ``good-$\lambda$" inequality, which implies the final result.

 This paper is organized as follows. In Section 2, we prove Theorem \ref{thm0} by deriving the comparison estimates and improving the regularity of solutions to the homogeneous system, in both interior and boundary cases. In Section 3, we apply Theorem \ref{thm0} to prove the ``good-$\lambda$" inequality (Lemma \ref{cover}) and finish the proof of Theorem \ref{thm1}.

\section{Proof of Theorem \ref{thm0}}
\subsection{Interior estimates}
For each ball $B_{2R}(x_0)\subset\subset\Omega$, one considers the unique solution
\begin{equation*}
	w\in L^p((-\lambda^{p-2}(2R)^2+t_0,t_0), W_{0}^{1,p}(B_{2R}(x_0)))+u
\end{equation*}
to the following equation
\begin{equation}\label{e2}
	\left\{
	\begin{array}
		{ll}%
		w_t -\operatorname{div}((a(t))_{B_{2R}(x_0)}|\nabla w|^{p-2}\nabla w) =0&\text{in }\ Q_{2R}^\lambda(x_0,t_0),\\
		w=u&\text{on}\
		\partial_{p} Q_{2R}^\lambda(x_0,t_0).
		\\
	\end{array}
	\right.
\end{equation}
Here and thoughout the paper, we define \[(a(t))_{B_{2R}(x_0)}=\frac{1}{|B_{2R}(x_0)|}\int_{B_{2R}(x_0)} a(x,t) dx.\]
For simplicity, we will always assume $(x_0,t_0)=(0,0)$ in our proof, and denote 
$$Q_{R}^\lambda=Q_{R}^\lambda(0,0),\quad\quad B_R=B_R(0),$$
for any $\lambda,R>0$. 
The following lemma gives an estimate for $\nabla u-\nabla w$ in $Q^\lambda_{2R}(x_0,t_0)$.
\begin{lemma}\label{u-w}
	Let $w$ be a weak solution of \eqref{e2}. For any $\vartheta>\frac{{2n}}{p}-(2+n)$ and any $\eta, \kappa_0\in(0,1)$, 
	let 
	\begin{align*}
		\sigma = p+\vartheta +\frac{p(\vartheta +2)}{n}>\frac{n(2-p)}{p},\ \sigma_1 = p+\frac{n\vartheta }{n+\vartheta +2}>\frac{n(2-p)}{2}.
	\end{align*}
	If  $\nabla u\in L^{\sigma_1+\kappa_0}(Q_{2R}^\lambda(x_0,t_0))$, and  $F\in L^{\sigma_1}(Q_{2R}^\lambda(x_0,t_0))$, and
	\begin{equation}\label{coa}
		[a]_{2R}\leq \eta^{\frac{n^4}{\kappa_0^4(p-1)^4}} ,
	\end{equation}  then we have\\
	(1) 
		\begin{equation}\label{lem1-case3}
		\begin{split}
			& \left(\fint_{Q_{2R}^\lambda(x_0,t_0)} |\lambda \nabla(u-w)|^p\right)^{\frac{1}{p}}
			\leq C(\eta)\left(\fint_{Q_{2R}^\lambda(x_0,t_0)}  |\lambda  F|^{p}\right)^\frac{1}{p}+\eta\left(\fint_{Q_{2R}^\lambda(x_0,t_0)}|\lambda \nabla u|^{\sigma_1}\right)^\frac{1}{\sigma_1}.
		\end{split}
	\end{equation}	
	(2)
	\begin{equation}\label{lem1-case1}
	\begin{split}
		& \left[\fint_{t_0-\lambda^{p-2}
			(2R)^2}^{t_0}\left(\fint_{B_{2R}(x_0)}
		\left(\frac{\lambda}{R}|u-w|\right)^{\sigma}dx\right)^\frac{2}{\sigma}dt\right]
		^\frac{1}{2}\\
		&\quad\quad\leq
		C(\eta,\kappa_0)\left(\fint_{Q_{2R}^\lambda(x_0,t_0)}|\lambda F|^{\sigma_1}\right)^{\frac{n+p}{\sigma_1(n+2)}}+
		\eta\left(\fint_{Q_{2R}^\lambda(x_0,t_0)}|\lambda \nabla u|^{\sigma_1+\kappa_0}\right)^{\frac{n+p}{(\sigma_1+\kappa_0)(n+2)}}.
	\end{split}
\end{equation}
\end{lemma}
\begin{proof} For simplicity, let $(x_0,t_0)=(0,0)$.  Set $\mathbf{f}=u-w$. 
	The equation of $\mathbf{f}$ reads as
	\begin{equation}\label{equ-w}
		\mathbf{f}_t-\operatorname{div} ( (a(t))_{B_{2R}}(|\nabla u|^{p-2}\nabla u-|\nabla w|^{p-2}\nabla w))=\operatorname{div} (	\mathcal{G}),
	\end{equation}
	where 
	\begin{equation*}
		\mathcal{G}=|F|^{p-2}F+(a(t)-(a(t))_{B_{2R}})|\nabla u|^{p-2}\nabla u.
	\end{equation*}
We first prove \eqref{lem1-case3}. 	Note that
$$
\big( |\nabla u|^{p-2}\nabla u-|\nabla w|^{p-2}\nabla w\big): \nabla\mathbf{f}\gtrsim |\nabla \mathbf{f}|^2(|\nabla u|+|\nabla w|)^{p-2}.
$$
Denote
\begin{equation*}
	g=|\nabla \mathbf{f}|^{2}(|\nabla u|+|\nabla w|)^{p-2}.
\end{equation*}
Using $\mathbf{f}$ as test function in \eqref{equ-w}, one has
\begin{align*}
	\frac{1}{2}\frac{d}{dt} \int_{B_{2R}} |\mathbf{f}|^2 +c_a\int_{B_{2R}} g  \leq C \int_{B_{2R}} |	\mathcal{G}| |\nabla\mathbf{f}|.	\end{align*}
Integrating in time and using H\"older's inequality we obtain
\begin{align*}
	&\fint_{Q_{2R}^\lambda} g \lesssim \fint_{Q_{2R}^\lambda}  |	\mathcal{G}| |\nabla \mathbf{f}|\lesssim \left(\fint_{Q_{2R}^\lambda}  |	\mathcal{G}|^\frac{p}{p-1}\right)^\frac{p-1}{p}
	\left(\fint_{Q_{2R}^\lambda}  |\nabla \mathbf{f}|^{p}\right)^\frac{1}{p}.
\end{align*}
By \eqref{G} and  H\"older's inequality we have 
\begin{align*}
	\left(\fint_{Q_{2R}^\lambda}  |	\mathcal{G}|^\frac{p}{p-1}\right)^\frac{p-1}{p}
	&\lesssim \left(\fint_{Q_{2R}^\lambda}  |F|^{p}\right)^\frac{p-1}{p}+\left(\fint_{Q_{2R}^\lambda} |\nabla u|^{p}|a(t)-(a(t))_{B_{2R}}|^{\frac{p}{p-1}}\right)^\frac{p-1}{p}\\
	&\lesssim \left(\fint_{Q_{2R}^\lambda}  |F|^{p}\right)^\frac{p-1}{p}+C_a^{1-\frac{
			(\sigma_1-p)(p-1)}{\sigma_1p}}[a]_{2R}^\frac{
		(\sigma_1-p)(p-1)}{\sigma_1p}\left(\fint_{Q_{2R}^\lambda}|\nabla u|^{\sigma_1}\right)^\frac{p-1}{\sigma_1}\\
	&\lesssim  \left(\fint_{Q_{2R}^\lambda}  |F|^{p}\right)^\frac{p-1}{p}+\eta\left(\fint_{Q_{2R}^\lambda}|\nabla u|^{\sigma_1}\right)^\frac{p-1}{\sigma_1}.
\end{align*}
	Observe that
\begin{equation*}
	|\nabla \mathbf{f}|^p\leq g+g^{p/2}\min\{|\nabla u|,|\nabla w|\}^{p(1-p/2)}.
\end{equation*}
By H\"older's inequality
\begin{align*}
	\fint_{Q_{2R}^\lambda} |\nabla\mathbf{f}|^p\leq& \fint_{Q_{2R}^\lambda}g+\left(\fint_{Q_{2R}^\lambda} g \right)^\frac{p}{2}\left(\fint_{Q_{2R}^\lambda}|\nabla u|^p \right)^\frac{2-p}{2} \\
	\lesssim &\left\{\left(\fint_{Q_{2R}^\lambda}  |F|^{p}\right)^\frac{p-1}{p}+\eta\left(\fint_{Q_{2R}^\lambda}|\nabla u|^{\sigma_1}\right)^\frac{p-1}{\sigma_1} \right\}
	\left(\fint_{Q_{2R}^\lambda}  |\nabla \mathbf{f}|^{p}\right)^\frac{1}{p}\\
	&+\left\{\left(\fint_{Q_{2R}^\lambda}  |F|^{p}\right)^\frac{p-1}{2}+\eta^\frac{p}{2}\left(\fint_{Q_{2R}^\lambda}|\nabla u|^{\sigma_1}\right)^\frac{p(p-1)}{2\sigma_1} \right\}
	\left(\fint_{Q_{2R}^\lambda}  |\nabla \mathbf{f}|^{p}\right)^\frac{1}{2}\left(\fint_{Q_{2R}^\lambda}|\nabla u|^p \right)^\frac{2-p}{2}.
\end{align*}
Applying Young's inequality, we have \eqref{lem1-case3}.\vspace{0.3cm}\\
Then we prove \eqref{lem1-case1}. 
	Using $|\mathbf{f}|^{\vartheta}\mathbf{f}$, $\vartheta>0$ as test function in \eqref{equ-w} and doing integration by parts yields
	\begin{equation}\label{intx}
		\begin{split}
			&\frac{d}{dt}\int_{B_{2R}} |\mathbf{f}|^{\vartheta+2} +\int_{B_{2R}}  (|\nabla u|^{p-2} \nabla u-|\nabla w|^{p-2}\nabla w ):\nabla (|\mathbf{f}|^\vartheta\mathbf{f})
			\lesssim  \int_{B_{2R}} |\mathbf{f}|^\vartheta | \nabla\mathbf{f}| |	\mathcal{G}|.
		\end{split}
	\end{equation}
Note that for any matrices $M_1, M_2\in\mathbb{R}^{N\times n}$, one has
\begin{align*}
	||M_1|^{p-2}M_1-|M_1-M_2|^{p-2}(M_1-M_2)|\lesssim |M_2|^{p-1}.
\end{align*}
Moreover, since $Id+\vartheta\frac{\mathbf{f}\otimes \mathbf{f}}{|\mathbf{f}|^2}$ is positive definite for $\vartheta>-1$, we have
\begin{align*}
	|\nabla \mathbf{f}|^{p-2}\nabla \mathbf{f} : \nabla\left( |\mathbf{f}|^{\vartheta}\mathbf{f}\right)=\sum_{i=1}^n|\nabla \mathbf{f}|^{p-2}|\mathbf{f}|^\vartheta (\partial_i \mathbf{f})^T\left(Id+\vartheta\frac{\mathbf{f}\otimes \mathbf{f}}{|\mathbf{f}|^2}\right)\partial_i \mathbf{f}\gtrsim |\mathbf{f}|^\vartheta |\nabla \mathbf{f}|^{p}.
\end{align*}
Thus,
\begin{align*}
	&(|\nabla u|^{p-2} \nabla u-|\nabla w|^{p-2}\nabla w ): \nabla\left( |\mathbf{f}|^{\vartheta}\mathbf{f}\right)
	\\&	= |\nabla \mathbf{f}|^{p-2}\nabla \mathbf{f} : \nabla\left( |\mathbf{f}|^{\vartheta}\mathbf{f}\right)+ (|\nabla u|^{p-2} \nabla u-|\nabla w|^{p-2}\nabla w -|\nabla \mathbf{f}|^{p-2}\nabla \mathbf{f}): \nabla\left( |\mathbf{f}|^{\vartheta}\mathbf{f}\right)\\&
	\geq C^{-1} |\mathbf{f}|^\vartheta |\nabla \mathbf{f}|^{p}-C|\nabla u|^{p-1}|\nabla \mathbf{f}||\mathbf{f}|^\vartheta.
\end{align*}
Then \eqref{intx} leads to 
	\begin{equation*}
	\begin{split}
		&\frac{d}{dt}\int_{B_{2R}} |\mathbf{f}|^{\vartheta+2} +\int_{B_{2R}}   |\mathbf{f}|^\vartheta |\nabla \mathbf{f}|^{p}
		\lesssim  \int_{B_{2R}} |\mathbf{f}|^\vartheta | \nabla\mathbf{f}| |	\mathcal{G}|+\int_{B_{2R}} |\mathbf{f}|^\vartheta | \nabla\mathbf{f}| |\nabla u|^{p-1}.
	\end{split}
\end{equation*}
Integrate in time we obtain
	\begin{equation}\label{esu-w}
		\sup_{-\lambda^{p-2}(2R)^2\leq t\leq 0} \int_{B_{2R}} |\mathbf{f}|^{\vartheta+2}dx+\int_{Q_{2R}^\lambda} |\mathbf{f}|^\vartheta |\nabla \mathbf{f}|^{p}
		\lesssim \int_{Q_{2R}^\lambda} |\mathbf{f}|^\vartheta |\nabla \mathbf{f}| |	\mathcal{G}|+\int_{Q_{2R}^\lambda} |\mathbf{f}|^\vartheta |\nabla \mathbf{f}| |\nabla u|^{p-1}.
	\end{equation}
	Denote
\begin{equation*}
	\begin{split}
	&	E_1=\int_{Q_{2R}^\lambda} |\mathbf{f}|^\vartheta |\nabla \mathbf{f}| |	\mathcal{G}|,\\
	&	 E_3=\sup_{-\lambda^{p-2}(2R)^2\leq t\leq 0} \int_{B_{2R}} |\mathbf{f}|^{\vartheta+2}dx,  
	\end{split}\quad\quad\quad
	\begin{split}
 & E_2=\int_{Q_{2R}^\lambda} |\mathbf{f}|^\vartheta |\nabla \mathbf{f}| |\nabla u|^{p-1},\\
	  &E_4=	\int_{Q_{2R}^\lambda} |\mathbf{f}|^\vartheta 	|\nabla \mathbf{f}|^p.
\end{split}
\end{equation*}
	It follows from \eqref{esu-w} that
	\begin{equation}\label{E123}
		E_3+ E_4\lesssim E_1+E_2.
	\end{equation}
	Next we estimate the right hand side terms $E_1$ and $E_2$. Using H\"older's inequality with exponents   $s'$, $p$, $\frac{sp}{p-s}$ satisfying
	\begin{equation*}
		s=\frac{p(n+\vartheta+2)+\vartheta}{(n+\vartheta+2)+\vartheta},\ \quad\frac{1}{s}+\frac{1}{s'}=1,
	\end{equation*}
	one gets
	\begin{equation*}
		\begin{split}
			E_1
			\leq& \left(\int_{Q_{2R}^\lambda}  |	\mathcal{G}|^{s'}  \right)^{\frac{1}{s'}}
			\left(\int_{Q_{2R}^\lambda}  |\mathbf{f}|^{\vartheta} |\nabla \mathbf{f}|^{p}  \right)^{\frac{1}{p}}\left(\int_{Q_{2R}^\lambda}  |\mathbf{f}|^{\frac{\vartheta(p-1)}{p}\cdot\frac{sp}{p-s}} \right)^{\frac{p-s}{sp}}\\
			\leq &\left(\int_{Q_{2R}^\lambda}  |	\mathcal{G}|^{s'}  \right)^{\frac{1}{s'}}E_4^\frac{1}{p}
			\|\mathbf{f}\|_{L^{\sigma} }^{\frac{\vartheta(p-1)}{p}},
		\end{split}
	\end{equation*}
	where we used the fact that 
	\begin{align*}
	\frac{\vartheta s(p-1)}{p-s}=\frac{p(\vartheta+2)}{n}+p+\vartheta=\sigma.
	\end{align*}
	Note that 
	\begin{align}\label{G}
		|	\mathcal{G}|\leq |F|^{p-1}+|\nabla u|^{p-1}|a(t)-(a(t))_{B_{2R}}|.
	\end{align}
	By H\"{o}lder's inequality  we obtain for any $\kappa_0>0$
	\begin{align*}
		\left(\int_{Q_{2R}^\lambda}  |	\mathcal{G}|^{s'}  \right)^{\frac{1}{s'}}
		&\leq \|F\|^{p-1}_{L^{\sigma_1} }+\left(\int_{Q_{2R}^\lambda} |\nabla u|^{\sigma_1}|a(t)-(a(t))_{B_{2R}}|^{s'}\right)^\frac{1}{s'}\\
		&\lesssim \|F\|^{p-1}_{L^{\sigma_1} }+C_a^{1-\frac{\kappa_0}{(\sigma_1+\kappa_0)s'}}\left(\int_{Q_{2R}^\lambda}|a(t)-(a(t))_{B_{2R}}|\right)^\frac{\kappa_0}{(\sigma_1+\kappa_0)s'}\left(\int_{Q_{2R}^\lambda}|\nabla u|^{\sigma_1+\kappa_0}\right)^\frac{\sigma_1}{(\sigma_1+\kappa_0)s'},
	\end{align*}
where we used the fact that
\begin{align*}
		s'(p-1)=p+\frac{n\vartheta}{n+\vartheta+2}=\sigma_1.
\end{align*}
	Note that 
	$$
	\left(\int_{Q_{2R}^\lambda}|a(x,t)-(a(t))_{B_{2R}}|\right)^\frac{\kappa_0}{(\sigma_1+\kappa_0)s'}\leq ([a]_{2R}|Q_{2R}^\lambda|)^\frac{\kappa_0(p-1)}{\sigma_1(\sigma_1+\kappa_0)}\overset{\eqref{coa}}\leq \eta^2|Q_{2R}^\lambda|^\frac{\kappa_0(p-1)}{\sigma_1(\sigma_1+\kappa_0)}.
	$$
	Hence 
	\begin{equation}\label{rhs}
		\begin{split}
			E_1
			\lesssim &\left(\|F\|^{p-1}_{L^{\sigma_1} }+\eta^2|Q_{2R}^\lambda|^\frac{\kappa_0(p-1)}{\sigma_1(\sigma_1+\kappa_0)}\|\nabla  u\|_{L^{\sigma_1+\kappa_0}}^{p-1}\right)E_4^\frac{1}{p}
			\|\mathbf{f}\|_{L^{\sigma} }^{\frac{\vartheta(p-1)}{p}}.
		\end{split}
	\end{equation}
	Now we estimate $\|\mathbf{f}\|_{L^{\sigma} (Q_{2R}^\lambda)}$ in \eqref{rhs}. The interpolation inequality shows that
	\begin{align}\label{Int}
		\int_{B_{2R}}|\mathbf{f}|^\sigma 
		\lesssim &\left(\int_{B_{2R}}|\mathbf{f}|^{\vartheta+2} \right)^{\frac{p}{n}}
		\left(\int_{B_{2R}}|\mathbf{f}|^{\frac{n(p+\vartheta)}{n-p}} \right)^{\frac{n-p}{n}}.
	\end{align}
	Thanks to the Sobolev embedding inequality,
	we have
	\begin{align*}
		\left(\int_{B_{2R}}|\mathbf{f}|^\frac{n(p+\vartheta)}{n-p} \right)^{\frac{n-p}{n}}
		\lesssim&\left(\int_{B_{2R}}\left||\mathbf{f}|^{\frac{\vartheta}{p}}\mathbf{f}\right|^\frac{np}{n-p} \right)^{\frac{n-p}{n}}
		\lesssim\int_{B_{2R}}\left|\nabla(|\mathbf{f}|^{\frac{\vartheta}{p}}\mathbf{f})\right|^p 
		=c\int_{B_{2R}} |\mathbf{f}|^\vartheta|\nabla \mathbf{f}|^p.
	\end{align*}
	Hence
	\begin{align}\label{Sob}
		\int_{-\lambda^{p-2} (2R)^2}^{0}\left(\int_{B_{2R}}|\mathbf{f}|^\frac{n(p+\vartheta)}{n-p} dx\right)^{\frac{n-p}{n}}dt
		\lesssim E_4.
	\end{align}
	Integrating over $t$ in \eqref{Int} and using \eqref{Sob} one has
	\begin{equation}\label{Lb}
		\begin{split}
			\int_{Q_{2R}^\lambda} |\mathbf{f}|^{\sigma}
			\lesssim&\sup_{t\in [-\lambda^{p-2}(2R)^2,0]}\left(\int_{B_{2R}}|\mathbf{f}|^{\vartheta+2} dx\right)^{\frac{p}{n}}
			\int_{-\lambda^{p-2} (2R)^2}^{0}\left(\int_{B_{2R}}|\mathbf{f}|^{\frac{n(p+\vartheta)}{n-p}} dx\right)^{\frac{n-p}{n}}dt\lesssim  E_3^{\frac{p}{n}}E_4.
		\end{split}
	\end{equation}
	Combining \eqref{rhs} and \eqref{Lb}, we obtain
	\begin{equation}\label{es-E1}
		\begin{split}
			E_1\lesssim &\left(\|F\|^{p-1}_{L^{\sigma_1} }+\eta^2|Q_{2R}^\lambda|^\frac{\kappa_0(p-1)}{\sigma_1(\sigma_1+\kappa_0)}\|\nabla  u\|_{L^{\sigma_1+\kappa_0}}^{p-1}\right)E_3^\frac{\vartheta(p-1)}{n\sigma} E_4^{\frac{1}{p}+\frac{\vartheta(p-1)}{p\sigma}}.
		\end{split}
	\end{equation}
Finally, we estimate $E_2$. It is easy to check that 
$$
\frac{p-1}{\sigma_1}+\frac{1}{p}+\frac{\vartheta(p-1)}{\sigma p}=1.
$$
Moreover, we can take $\epsilon_0$ small enough such that there exists $0<\kappa_0'<\kappa_0$ satisfying 
$$
\frac{p-1}{\sigma_1+\kappa_0'}+\frac{1}{p}+\frac{\vartheta(p-1+\epsilon_0)}{\sigma p}=1.
$$
 Applying H\"older's inequality with exponents  $\frac{\sigma_1+\kappa_0'}{p-1}$, $\frac{p}{1-\epsilon_0}$, $\frac{p}{\epsilon_0}$, $\frac{\sigma p}{\vartheta(p-1+\epsilon_0)}$, one has
\begin{align*}
	E_2\lesssim& \left(\int_{Q_{2R}^\lambda}  |	\nabla u|^{\sigma_1+\kappa_0'}  \right)^{\frac{p-1}{\sigma_1+\kappa_0'}}
	\left(\int_{Q_{2R}^\lambda}  |\mathbf{f}|^{\vartheta} |\nabla \mathbf{f}|^{p}  \right)^{\frac{1-\epsilon_0}{p}}\left(\int_{Q_{2R}^\lambda}|\nabla \mathbf{f}|^p\right)^\frac{\epsilon_0}{p}\left(\int_{Q_{2R}^\lambda}  |\mathbf{f}|^\sigma \right)^{\frac{\vartheta(p-1+\epsilon_0)}{\sigma p}}\\
	\lesssim &\|\nabla u\|_{L^{\sigma_1+\kappa_0'}}^{p-1}E_4^\frac{1-\epsilon_0}{p}\|\nabla \mathbf{f}\|_{L^p}^{\epsilon_0}
	\|\mathbf{f}\|_{L^{\sigma} }^{\frac{\vartheta(p-1+\epsilon_0)}{p}}.
\end{align*}
By  \eqref{lem1-case3} and \eqref{Lb} , we have 
\begin{align*}
	E_2\lesssim& \|\nabla u\|_{L^{\sigma_1+\kappa_0'}}^{p-1}\left(C(\eta)\|F\|_{L^p}^{\epsilon_0}+\eta^2|Q^\lambda_{2R}|^{\frac{\epsilon_0(\sigma_1-p)}{\sigma_1p}} \|\nabla u\|_{L^{\sigma_1}}^{\epsilon_0}\right)
	E_3^{\frac{\vartheta(p-1+\epsilon_0)}{n\sigma}}E_4^{\frac{\vartheta(p-1+\epsilon_0)}{p\sigma}+\frac{1-\epsilon_0}{p}}.
\end{align*}
Combining this with \eqref{E123} and \eqref{es-E1}, we get
	\begin{equation*}
		\begin{split}
			E_3+E_4\lesssim &\left(\|F\|^{p-1}_{L^{\sigma_1} }+\eta^2|Q_{2R}^\lambda|^\frac{\kappa_0(p-1)}{\sigma_1(\sigma_1+\kappa_0)}\|\nabla  u\|_{L^{\sigma_1+\kappa_0}}^{p-1}\right)E_3^\frac{\vartheta(p-1)}{n\sigma} E_4^{\frac{1}{p}+\frac{\vartheta(p-1)}{p\sigma}}\\
			&\quad+\|\nabla u\|_{L^{\sigma_1+\kappa_0'}}^{p-1}\left(C(\eta)\|F\|_{L^p}^{\epsilon_0}+\eta^2 |Q^\lambda_{2R}|^{\frac{\epsilon_0(\sigma_1-p)}{\sigma_1p}}\|\nabla u\|_{L^{\sigma_1}}^{\epsilon_0}\right)
		E_3^{\frac{\vartheta(p-1+\epsilon_0)}{n\sigma}}E_4^{\frac{\vartheta(p-1+\epsilon_0)}{p\sigma}+\frac{1-\epsilon_0}{p}}.
		\end{split}
	\end{equation*}
By	H\"{o}lder's inequality	and Young's inequality, we obtain
	\begin{equation*}
		\begin{split}
		E_3+E_4
			\lesssim &\|F\|_{L^{\sigma_1} }^\frac{\sigma n}{n+2}+C(\eta)\|F\|_{L^p}^\frac{\sigma n}{n+2}+\eta|Q^\lambda_{2R}|^\frac{\kappa_0\sigma n}{\sigma_1(\sigma_1+\kappa_0)(n+2)}\|\nabla  u\|_{L^{\sigma_1+\kappa_0}}^\frac{\sigma n}{n+2}\\
			&+\eta|Q^\lambda_{2R}|^{\frac{\sigma n\epsilon_0(\sigma_1-p)}{\sigma_1p(n+2)(p-1+\epsilon_0)}}\|\nabla u\|_{L^{\sigma_1+\kappa_0'}}^\frac{(p-1)\sigma n}{(n+2)(p-1+\epsilon_0)}\|\nabla u\|_{L^{\sigma_1}}^\frac{\epsilon_0\sigma n}{(n+2)(p-1+\epsilon_0)}\\
			\lesssim &\|F\|_{L^{\sigma_1} }^\frac{\sigma n}{n+2}+C(\eta)\|F\|_{L^p}^\frac{\sigma n}{n+2}+\eta|Q^\lambda_{2R}|^\frac{\kappa_0\sigma n}{\sigma_1(\sigma_1+\kappa_0)(n+2)}\|\nabla  u\|_{L^{\sigma_1+\kappa_0}}^\frac{\sigma n}{n+2}.
		\end{split}
	\end{equation*}
	%Combining this with \eqref{Lb}, we obtain
	%\begin{equation}\nonumber
	%\begin{split} \left(\fint_{Q_{2R}^\lambda}\left(\frac{\lambda}{R}|\mathbf{f}|\right)^{\sigma }\right)^{\frac{1}{\sigma }}\leq
	%&C\left(\fint_{Q_{2R}^\lambda}|\lambda F|^{\sigma_1}\right)^{\frac{n+p}{\sigma_1(n+2)}}+C
	%\left(\fint_{Q_{2R}^\lambda}|\lambda F|^{\sigma_1}\right)^{\frac{(p-1)(n+p)}{\sigma_1(n+2)}}
	%\left(\fint_{Q_{2R}^\lambda}|\lambda \nabla u|^{\sigma_1}\right)^{\frac{(2-p)(n+p)}{\sigma_1(n+2)}}
	%\\
	%&+
	%\varepsilon_a^\frac{\kappa_0(n+p)}{\sigma_1(\sigma_1+\kappa_0)(n+2)}\left(\fint_{Q_{2R}^\lambda}|\lambda \nabla u|^{\sigma_1+\kappa_0}\right)^{\frac{n+p}{(\sigma_1+\kappa_0)(n+2)}}\\
	%&+\varepsilon_a^\frac{\kappa_0 (n+p)(p-1)}{\sigma_1(\sigma_1+\kappa_0)(n+2)}\left(\fint_{Q_{2R}^\lambda}|\lambda \nabla u|^{\sigma_1+\kappa_0}\right)^{\frac{(n+p)(p-1)}{(\sigma_1+\kappa_0)(n+2)}}
	%\left(\fint_{Q_{2R}^\lambda}|\lambda \nabla u|^{\sigma_1}\right)^{\frac{(2-p)(n+p)}{\sigma_1(n+2)}}.
	%\end{split}
	%\end{equation}
	Combining this with \eqref{Lb}, we obtain
	\begin{equation}\nonumber
		\begin{split} \left(\fint_{Q_{2R}^\lambda}\left(\frac{\lambda}{R}|\mathbf{f}|\right)^{\sigma }\right)^{\frac{1}{\sigma }}\lesssim
			&C(\eta,\kappa_0)\left(\fint_{Q_{2R}^\lambda}|\lambda F|^{\sigma_1}\right)^{\frac{n+p}{\sigma_1(n+2)}}+
			\eta\left(\fint_{Q_{2R}^\lambda}|\lambda \nabla u|^{\sigma_1+\kappa_0}\right)^{\frac{n+p}{(\sigma_1+\kappa_0)(n+2)}}.
		\end{split}
	\end{equation}
	Note that $\vartheta>\frac{{2n}}{p}-(2+n)$ and $1<p\leq \frac{2n}{n+2}$ imply that
	\begin{align*}
		\sigma  = p+\vartheta+\frac{p(\vartheta+2)}{n}>\frac{n(2-p)}{p}\geq 2.
	\end{align*}
	We can use H\"older's inequality to get \eqref{lem1-case1}. This completes the proof.
	\vspace{0.5cm}

\end{proof}
\begin{remark}\label{mixcomparison}
	In the above lemma, we see that $\sigma\to\frac{n(2-p)}{p}$ and $\sigma_1\to\frac{n(2-p)}{2}$ as $\vartheta\to\frac{{2n}}{p}-(2+n)$.
\end{remark}
\begin{remark} Consider $w$ as the unique weak solution to the following system
	\begin{equation*}
		\left\{
		\begin{array}
			{ll}%
			w_t -\operatorname{div}(a(x,t)|\nabla w|^{p-2}\nabla w) =0&\text{in }\ \Omega_T,\\
			w=0&\text{on}\
			\partial_{p} \Omega_T.
			\\
		\end{array}
		\right.
	\end{equation*}
	Then $w\equiv 0$ in $\Omega_T$. Following the proof of Lemma \ref{u-w}, one has for any $\vartheta>\frac{{2n}}{p}-(2+n)$
	\begin{equation*}
		\begin{split}
			& \left(\fint_{\Omega_T}
		|u|^{\sigma}\right)^\frac{1}{\sigma}\leq
			C\left(\fint_{\Omega_T}| F|^{\sigma_1}\right)^{\frac{n+p}{\sigma_1(n+2)}},
		\end{split}
	\end{equation*}
	where
	\begin{align*}
		\sigma = p+\vartheta+\frac{p(\vartheta+2)}{n}>\frac{n(2-p)}{p},\ \quad\sigma_1 = p+\frac{n\vartheta}{n+\vartheta+2}>\frac{n(2-p)}{2}.
	\end{align*}
\end{remark}
\begin{remark}
	Note that if $p>\frac{2n}{n+2}$, then the $L^p$ comparison estimate (\eqref{lem1-case3} in Lemma \ref{u-w}) is enough to obtain the final result.
\end{remark}

\begin{lemma}\label{lemw}
	Let $w$ be a weak solution of \eqref{e2}.
	Then for any $\beta_1>\frac{n(2-p)}{2}$, one has
	\begin{align}\label{Dwbeta}
		\fint_{Q^\lambda_{R}(x_0,t_0)}(\lambda |\nabla w|)^{\beta_1}
		\lesssim
		A
		^\frac{2\beta_1}{ p}+A+G_{\frac{2p-2}{p}}\left(\fint_{Q^\lambda_{2R}(x_0,t_0)}(\lambda|\nabla w|)^{p}\right).
	\end{align}
Here the function $G_{\frac{2p-2}{p}}$ is defined in \eqref{formularG} and
	\begin{equation}\label{def-A}
		A=\left[\fint_{t_0-\lambda^{p-2}
			(2R)^2}^{t_0}\left(\fint_{B_{2R}(x_0)}
		\left(\frac{\lambda}{R}|w-w_{2R}|\right)^{q_0}dx\right)^\frac{\sigma_2}{q_0}dt\right]
		^\frac{1}{\sigma_2}
	\end{equation}
	with $q_0=\frac{2\beta_1}{p}>\frac{n(2-p)}{p}$, $\sigma_2=\min\{\frac{2\beta_1(n+p)}{2\beta_1+(n+2)p},2\}$ and $w_{2R}=\fint_{B_{2R}(x_0)}wdx$.
\end{lemma}
\begin{proof} For simplicity, we assume that $(x_0,t_0)=(0,0)$.
	Let $v$ be the solution of
	\begin{equation}\label{eq-v}
		v_t -\operatorname{div}\left((a(t))_{B_{2R}}(|\nabla v|^2+\varepsilon^2)^{\frac{p-2}{2}}\nabla v\right)= 0 \quad \text{in}~ Q_{2R}^\lambda
	\end{equation}
and $v=w$ on $\partial_{p}Q_{2R}^\lambda$. By the classical regularity theory,  $v$ is smooth for any $\varepsilon>0$. Moreover, $v\to w$ in $C^1_{loc}(Q_{2R}^\lambda)$ as $\varepsilon\rightarrow 0$. 
	Let $R\leq R_2<R_1\leq 3R/2$ and $R_1-R_2\geq cR$ for some $c>0$. Define $\psi$ be the standard cutoff function such that $\psi=1$ in $Q^\lambda_{R_2}$, $\psi=0$ in a neighborhood of parabolic boundary of $Q^\lambda_{R_1}$, and
	\[0\leq \psi\leq 1,\ |\psi_t|\lesssim \frac{1}{\lambda^{p-2}(R_1-R_2)^2},\ |\nabla\psi|\lesssim \frac{1}{R_1-R_2}.\]
	Denote $\theta_0=\frac{2\sigma_2}{2+\sigma_2}\in(0,1]$. 	Let $k>\frac{1}{\theta_0}+1$ be an integer which will be decided later. For any $p-2< \mu$, one has
\begin{equation}\label{theta0}
	\begin{split}
			&\int_{Q^{\lambda}_{R_1}}
		(|\nabla v|^2+\varepsilon^2)^{\frac{\mu}{2}}|\nabla v|^2\psi^{k} \\&\lesssim  \int_{-\lambda^{p-2}R_1^2}^{0} \left(\int _{B_{R_1}} (|\nabla v|^2+\varepsilon^2)^{\frac{\mu}{2}}|\nabla v|^2\psi^{k+1}dx\right)^{1-\theta_0}  \left(\int_{B_{R_1}} (|\nabla v|^2+\varepsilon^2)^{\frac{\mu}{2}}|\nabla v|^2\psi^{k'} dx\right)^{\theta_0} dt\\
		&\lesssim \left(\sup _{-\lambda^{p-2}R_1^2 \leq t \leq 0}\int_{B_{R_1}}(|\nabla v|^2+\varepsilon^2)^{\frac{\mu+2}{2}} \psi^{k+1} d x\right)^{1-\theta_0}\int_{-\lambda^{p-2}R_1^2}^{0}\left(\int_{B_{R_1}} (|\nabla v|^2+\varepsilon^2)^{\frac{\mu}{2}}|\nabla v|^2\psi^{k'} dx\right)^{\theta_0} dt,
	\end{split}
\end{equation}
where $k'=k+1-\frac{1}{\theta_0}$. 
  Take $q_1= \frac{2(\mu+2)}{p}$, using integration by parts and applying H\"older's inequality we obtain
\begin{equation}\label{eqm}
	\begin{split}
		&\int_{B_{R_1}}(|\nabla v|^2+\varepsilon^2)^{\frac{\mu}{2}}|\nabla v|^{2} \psi^{k'} d x=-\int_{B_{R_1}} (v-v_{R_1}) \operatorname{div}\left((|\nabla v|^2+\varepsilon^2)^{\frac{\mu}{2}} \nabla v \psi^{k'}\right) d x \\
		&\qquad\qquad\lesssim \left(\int_{B_{R_1}}|v-v_{R_1}|^{q_1}dx\right)^\frac{1}{q_1}\left(\int_{B_{R_1}}\left((|\nabla v|^2+\varepsilon^2)^{\frac{\mu}{2}}\left|\nabla^{2} v\right| \psi^{k'}\right)^{q_1'} d x\right)^\frac{1}{q_1'}\\
		&\qquad\qquad\quad+\left(\int_{B_{R_1}}|v-v_{R_1}|^{q_1}dx\right)^\frac{1}{q_1} \left(\int_{B_{R_1}}\left((|\nabla v|^2+\varepsilon^2)^{\frac{\mu+1}{2}}|\nabla \psi| \psi^{k'-1}\right)^{q_1'} d x\right)^\frac{1}{q_1'},
	\end{split}
\end{equation}
where $q_1'$ satisfies $\frac{1}{q_1}+\frac{1}{q_1'}=1$. Take $k$ large enough such that $$\min\left\{\left(\frac{k+1}{2}-\frac{1}{\theta_0}\right)\frac{2(\mu+2)}{\mu+2-p},\left(k-\frac{1}{\theta_0}\right)\frac{\mu+2}{\mu+1}\right\}\geq k+1.$$ By H\"older's inequality we obtain
\begin{equation}\label{term1}
	\begin{split}
		&\left(\int_{B_{R_1}}\left((|\nabla v|^2+\varepsilon^2)^{\frac{\mu}{2}}\left|\nabla^{2} v\right| \psi^{k'}\right)^{q_1'} d x\right)^\frac{1}{q_1'}\\
		&= \left(\int_{B_{R_1}}\left[\left((|\nabla v|^2+\varepsilon^2)^{\frac{\mu+p-2}{4}}\left|\nabla^{2} v\right| \psi^\frac{k+1}{2}\right)\left((|\nabla v|^2+\varepsilon^2)^\frac{\mu-p+2}{4} \psi^{\frac{k+1}{2}-\frac{1}{\theta_0} }\right)\right]^{q_1'} d x\right)^\frac{1}{q_1'}\\
		%			&\lesssim \left(\int_{B_{R_1}}(|\nabla v|^2+\varepsilon^2)^{\frac{\mu+p-2}{2}}\left|\nabla^{2} v\right|^{2} \psi^{k+1} d x\right)^{\frac{1}{2}}\left(\int_{B_{R_1}}(|\nabla v|^2+\varepsilon^2)^\frac{\mu+2}{2} \psi^{(k-1) \frac{\mu+2}{\mu+2-p}} d x\right)^{\frac{\mu+2-p}{2(\mu+2)}}\\
		&\lesssim \left(\int_{B_{R_1}}(|\nabla v|^2+\varepsilon^2)^{\frac{\mu+p-2}{2}}\left|\nabla^{2} v\right|^{2} \psi^{k+1} d x\right)^{\frac{1}{2}}\left(\int_{B_{R_1}}(|\nabla v|^2+\varepsilon^2)^\frac{\mu+2}{2} \psi^{k+1} d x\right)^{\frac{\mu+2-p}{2(\mu+2)}},
	\end{split}
\end{equation}
and
\begin{equation}\label{term2}
	\begin{split}
		\left(\int_{B_{R_1}}\left((|\nabla v|^2+\varepsilon^2)^{\frac{\mu+1}{2}}|\nabla \psi| \psi^{k'-1}\right)^{q_1'} d x\right)^\frac{1}{q_1'}&\lesssim R^{\frac{n}{\mu+2}-\frac{n}{q_1}-1}\left(\int_{B_{R_1}}(|\nabla v|^2+\varepsilon^2)^\frac{\mu+2}{2}\psi^{(k'-1)\frac{\mu+2}{\mu+1}} d x\right)^{\frac{\mu+1}{\mu+2}}\\
		&\lesssim R^{\frac{n}{\mu+2}-\frac{n}{q_1}-1}\left(\int_{B_{R_1}}(|\nabla v|^2+\varepsilon^2)^\frac{\mu+2}{2}\psi^{k+1} d x\right)^{\frac{\mu+1}{\mu+2}}.
	\end{split}
\end{equation}
	Substitute \eqref{term1}, \eqref{term2} in \eqref{eqm} and integrate in time,
by H\"{o}lder's inequality we obtain
\begin{align*}
	&\int_{-\lambda^{p-2}R_1^2}^{0}\left(\int_{B_{R_1}} (|\nabla v|^2+\varepsilon^2)^{\frac{\mu}{2}}|\nabla v|^2\psi^{k'} dx\right)^{\theta_0} dt\\
	&\lesssim \left(\int_{-\lambda^{p-2}R_1^2}^{0} \left(\int_{B_{R_1}}|v-v_{R_1}|^{q_1}dx\right)^\frac{\sigma_2}{q_1}dt\right)^{\frac{\theta_0}{\sigma_2}}\left(\sup_{ -\lambda^{p-2}R_1^2\leq t\leq 0}\int_{B_{R_1}}(|\nabla v|^2+\varepsilon^2)^\frac{\mu+2}{2} \psi^{k+1} d x\right)^{\frac{(\mu+2-p)\theta_0}{2(\mu+2)}}\\ &
	\qquad\qquad\times\left(\int_{Q^{\lambda}_{R_1}}(|\nabla v|^2+\varepsilon^2)^{\frac{\mu+p-2}{2}}\left|\nabla^{2} v\right|^{2} \psi^{k+1} \right)^{\frac{\theta_0}{2}} \\
	&\qquad+\left(\int_{-\lambda^{p-2}R_1^2}^{0} \left(\int_{B_{R_1}}|v-v_{R_1}|^{q_1}dx\right)^\frac{\sigma_2}{q_1}dt\right)^{\frac{\theta_0}{\sigma_2}} \left(\sup_{ -\lambda^{p-2}R_1^2\leq t\leq 0}\int_{B_{R_1}}(|\nabla v|^2+\varepsilon^2)^{\frac{\mu+2}{2}}\psi^{k+1} d x\right)^{\frac{(\mu+1)\theta_0}{\mu+2}}\\
	&\qquad\qquad\times R^{({\frac{n}{\mu+2}-\frac{n}{q_1}-1})\theta_0}
	(\lambda^{p-2}R^2)^{\frac{\theta_0}{2}}.
\end{align*}
Combining the above estimate with \eqref{theta0}, one has
	\begin{equation}\label{estv}
		\begin{split}
		&\int_{Q^{\lambda}_{R_1}}(|\nabla v|^2+\varepsilon^2)^{\frac{\mu}{2}}|\nabla v|^2\psi^{k} \\
	&\lesssim \left(\int_{-\lambda^{p-2}R_1^2}^{0} \left(\int_{B_{R_1}}|v-v_{R_1}|^{q_1}dx\right)^\frac{\sigma_2}{q_1}dt\right)^{\frac{\theta_0}{\sigma_2}}\left(\int_{Q^{\lambda}_{R_1}}(|\nabla v|^2+\varepsilon^2)^{\frac{\mu+p-2}{2}}\left|\nabla^{2} v\right|^{2} \psi^{k+1} \right)^{\frac{\theta_0}{2}}\\ &
	\qquad\qquad\qquad\qquad\times \left(\sup_{ -\lambda^{p-2}R_1^2\leq t\leq 0}\int_{B_{R_1}}(|\nabla v|^2+\varepsilon^2)^\frac{\mu+2}{2} \psi^{k+1} d x\right)^{\frac{(\mu+2-p)\theta_0}{2(\mu+2)}+1-\theta_0}\\
	&\qquad+\frac{(\lambda^{p-2}R^2)^{\frac{\theta_0}{2}}}{R^{(1-\frac{n}{\mu+2}+\frac{n}{q_1})\theta_0}}\left(\int_{-\lambda^{p-2}R_1^2}^{0} \left(\int_{B_{R_1}}|v-v_{R_1}|^{q_1}dx\right)^\frac{\sigma_2}{q_1}dt\right)^{\frac{\theta_0}{\sigma_2}} \\
	&\qquad\qquad\qquad\qquad\times\left(\sup_{ -\lambda^{p-2}R_1^2\leq t\leq 0}\int_{B_{R_1}}(|\nabla v|^2+\varepsilon^2)^{\frac{\mu+2}{2}}\psi^{k+1} d x\right)^{\frac{(\mu+1)\theta_0}{\mu+2}+1-\theta_0}.
	\end{split}
\end{equation}
	By differentiating \eqref{eq-v} with respect to $x_l$ we have
	\begin{align}\label{eq-vxl}
		(v_{x_l})_t-\left((a(t))_{B_{2R}} b_{ij}(|\nabla v|^2+\varepsilon^2)^{\frac{p-2}{2}}v_{x_lx_j}\right)_{x_i}=0, \quad\text{where}\quad b_{ij}=\delta_{ij}+(p-2)\frac{v_{x_i}v_{x_j}}{|\nabla v|^2+\varepsilon^2}.
	\end{align}
	Applying $v_{x_l}(|\nabla v|^2+\varepsilon^2)^{\frac\mu{2}} \psi^{k+1}$ as a test function, we obtain
	\begin{align*}
		&\sup _{-\lambda^{p-2}R_1^2 \leq t \leq 0} \int_{B_{R_1}}(|\nabla v|^2+\varepsilon^2)^{\frac{\mu+2}{2}} \psi^{k+1} d x+\int_{Q^{\lambda}_{R_1}}(|\nabla v|^2+\varepsilon^2)^{\frac{\mu+p-2}{2}}\left|\nabla^{2} v\right|^{2} \psi^{k+1}  \\
		&\quad\quad\quad\quad\quad\quad\lesssim \frac{1}{\lambda^{p-2}R^2} \int_{Q^{\lambda}_{R_1}}(|\nabla v|^2+\varepsilon^2)^{\frac{\mu+2}{2}} \psi^{k} +\frac{1}{R^{2}} \int_{Q^{\lambda}_{R_1}}(|\nabla v|^2+\varepsilon^2)^{\frac{\mu+p}{2}} \psi^{k-1} .
	\end{align*}
Substituting the above in \eqref{estv} one has
	\begin{align*}
		&\int_{Q^{\lambda}_{R_1}}(|\nabla v|^2+\varepsilon^2)^{\frac{\mu}{2}}|\nabla v|^2\psi^{k}\\& \lesssim \left(\int_{-\lambda^{p-2}R_1^2}^{0} \left(\int_{B_{R_1}}|v-v_{R_1}|^{q_1}dx\right)^\frac{\sigma_2}{q_1}dt\right)^{\frac{\theta_0}{\sigma_2}} \\&\quad\quad\times\left\{\left(\frac{1}{\lambda^{p-2}R^2}\int_{Q^{\lambda}_{R_1}}(|\nabla v|^2+\varepsilon^2)^{\frac{\mu+2}{2}} \psi^{k}+\frac{1}{R^2}\int_{Q^{\lambda}_{R_1}}(|\nabla v|^2+\varepsilon^2)^{\frac{\mu+p}{2}} \psi^{k-1}\right)^{1-\frac{p\theta_0}{2(\mu+2)}}\right.\\
		&\quad\quad\quad+\left. 	\frac{(\lambda^{p-2}R^2)^{\frac{\theta_0}{2}}}{R^{(1-\frac{n}{\mu+2}+\frac{n}{q_1})\theta_0}}\left(\frac{1}{\lambda^{p-2}R^2}\int_{Q^{\lambda}_{R_1}}(|\nabla v|^2+\varepsilon^2)^{\frac{\mu+2}{2}} \psi^{k}+\frac{1}{R^2}\int_{Q^{\lambda}_{R_1}}(|\nabla v|^2+\varepsilon^2)^{\frac{\mu+p}{2}} \psi^{k-1}\right)^{1-\frac{\theta_0}{\mu+2}}\right\}.
	\end{align*}
Let $\varepsilon\rightarrow 0$, then $v\to w$ in $C^1(Q_{3R/2}^\lambda)$. Note that $|v-v_{R_1}|\lesssim |w-w_{R_1}|$ as $\varepsilon\to 0$. The above estimate leads to 
\begin{align*}
		\int_{Q^{\lambda}_{R_1}}|\nabla w|^{\mu+2}\psi^{k}\lesssim& \left(\int_{-\lambda^{p-2}R_1^2}^{0} \left(\int_{B_{R_1}}|w-w_{R_1}|^{q_1}dx\right)^\frac{\sigma_2}{q_1}dt\right)^{\frac{\theta_0}{\sigma_2}} \\&\quad\times\left\{\left(\frac{1}{\lambda^{p-2}R^2}\int_{Q^{\lambda}_{R_1}}|\nabla w|^{\mu+2} \psi^{k}+\frac{1}{R^2}\int_{Q^{\lambda}_{R_1}}|\nabla w|^{\mu+p} \psi^{k-1}\right)^{1-\frac{p\theta_0}{2(\mu+2)}}\right.\\
	&\quad\quad+\left. \frac{(\lambda^{p-2}R^2)^{\frac{\theta_0}{2}}}{R^{(1-\frac{n}{\mu+2}+\frac{n}{q_1})\theta_0}}	\left(\frac{1}{\lambda^{p-2}R^2}\int_{Q^{\lambda}_{R_1}}|\nabla w|^{\mu+2} \psi^{k}+\frac{1}{R^2}\int_{Q^{\lambda}_{R_1}}|\nabla w|^{\mu+p} \psi^{k-1}\right)^{1-\frac{\theta_0}{\mu+2}}\right\}.
\end{align*}
	By Young's inequality we obtain
	\begin{equation*}
		\begin{split}
			\fint_{Q^\lambda_{R_2}}|\lambda\nabla w|^{\mu+2}
			&\lesssim  \sup_{p_1\in\{p,2\}}\left[\fint_{-\lambda^{p-2}R_1^2}^0\left(\fint_{B_{R_1}}\left(\frac{\lambda}{R_1}|w-w_{R_1}|\right)^{q_1}dx\right)^\frac{\sigma_2}{q_1}dt\right]^\frac{2(\mu+2)}{p_1\sigma_2}+\fint_{Q^\lambda_{R_1}}|\lambda\nabla w|^{\mu+p}.
		\end{split}
	\end{equation*}
	Using the fact that for any  $q_1\leq q_0$
	\begin{align*}
		\left(\fint_{B_{R_1}}
		|w-w_{R_1}|^{q_1}dx\right)^\frac{1}{q_1}\leq C	\left(\fint_{B_{3R/2}}
		|w-w_{3R/2}|^{q_0}dx\right)^\frac{1}{q_0},
	\end{align*}
	we obtain that for any $\mu$ satisfies $\frac{2(\mu+2)}{p}\leq q_0$
	\begin{align}\label{es-wmu}
		\fint_{Q^\lambda_{R_2}}(\lambda|\nabla w|)^{\mu+2}
		\lesssim
		\tilde{A}_0
		^\frac{2(\mu+2)}{ p}+\tilde{A}_0+\fint_{Q^\lambda_{R_1}}(\lambda|\nabla w|)^{\mu+p},
	\end{align}
where 
\begin{equation*}
	\tilde{A}_0=\left[\fint_{-\lambda^{p-2}
		(3R/2)^2}^0\left(\fint_{B_{3R/2}}
	\left(\frac{\lambda}{R}|w-w_{3R/2}|\right)^{q_0}dx\right)^\frac{\sigma_2}{q_0}dt\right]
	^\frac{1}{\sigma_2}.
\end{equation*}
Moreover, if $0\leq\mu\leq \mu_0$, \eqref{es-wmu} leads  to 
	\begin{align}\label{es-wmu1}
	\fint_{Q^\lambda_{R_2}}(\lambda|\nabla w|)^{\mu+2}
	\lesssim
	\tilde{A}_0
	^\frac{2(\mu_0+2)}{ p}+\tilde{A}_0+\fint_{Q^\lambda_{R_1}}(\lambda|\nabla w|)^{\mu+p}.
\end{align}
	Let $\beta_1>\frac{n(2-p)}{2}$, and fix
	\begin{align*}
	q_0=\frac{2\beta_1}{p}>\frac{n(2-p)}{p}.
\end{align*}
	We will bound $\|\nabla w\|_{L^{\beta_1}(Q_{R}^\lambda)}$ by $\|\nabla w\|_{L^p(Q_{3R/2}^\lambda)}$ using iteration. 
	Define $\tau=\beta_1-\left[\frac{\beta_1-p}{2-p}\right](2-p)$. Note that $p\leq\tau< 2$ and $\tau=\beta_1$ if $\beta_1<2$.
	 We will first bound $\|\nabla w\|_{L^\tau(Q_{5R/4}^\lambda)}$ by $\|\nabla w\|_{L^p(Q_{3R/2}^\lambda)}$, and then bound $\|\nabla w\|_{L^{\beta_1}(Q_{R}^\lambda)}$ by $\|\nabla w\|_{L^\tau(Q_{5R/4}^\lambda)}$.
	 By H\"{o}lder's inequality we have 
	\begin{align}\label{step0}
			\fint_{Q^\lambda_{5R/4}}(\lambda|\nabla w|)^{\tau}
		\lesssim
		\tilde{A}_0
		^\frac{2\tau}{ p}+\tilde{A}_0+\left(\fint_{Q^\lambda_{2R}}(\lambda|\nabla w|)^{\tau-2+p}\right)\lesssim 	\tilde{A}_0
		^\frac{2\beta_1}{ p}+\tilde{A}_0+\left(\fint_{Q^\lambda_{2R}}(\lambda|\nabla w|)^{p}\right)^\frac{\tau-2+p}{p},
	\end{align} 
where we also used \eqref{es-wmu} with $\mu=\tau-2$.
If $\beta_1<2$, then \eqref{Dwbeta} follows from the above estimate with $\tau=\beta_1$. It remains to consider $\beta_1\geq 2$. 
	Let 
	$N_0=\left[\frac{\beta_1-p}{2-p}\right]$, and fix $	\mu_0=\beta_1-2$ in \eqref{es-wmu1}. 
	Apply \eqref{es-wmu1} with
	\[\mu=\mu_j=\tau-2+(2-p)j,\quad R_1=R_{1j}=\frac{5R}{4}-\frac{R}{4N_0}(j-1),\ R_2=R_{0j}=\frac{5R}{4}-\frac{R}{4N_0}j,\]
	for $1\leq j\leq N_0$. We have $R_1-R_2=R/(4N_0)$ and $\mu_{N_0}+2= \beta_1$. After using \eqref{es-wmu1} $N_0$ steps we obtain 
	\begin{align*}
		\fint_{Q^\lambda_{R}}(\lambda |\nabla w|)^{\beta_1}
		\lesssim
		\tilde{A}_0
		^\frac{2\beta_1}{ p}+\tilde{A}_0+\fint_{Q^\lambda_{5R/4}}\left(\lambda|\nabla w|\right)^{\tau}\lesssim 	A
		^\frac{2\beta_1}{ p}+A+\fint_{Q^\lambda_{5R/4}}\left(\lambda|\nabla w|\right)^{\tau}.
	\end{align*}
Combining this with \eqref{step0} we obtain the result.
\end{proof}
  \begin{lemma} \label{z}
 	Let  $w$ be a weak solution of \eqref{e2}. Then for any  $\beta_{1}>\frac{n(2-p)}{2}$, there exist constants $\nu_0, \nu_1>0$ such that
 	\begin{equation}\label{ZA}
 		\lambda\|\nabla w\|_{L^\infty(Q_{R/2}^\lambda(x_0,t_0))}\leq C\left(\fint_{Q_{R}^\lambda(x_0,t_0)}|\lambda\nabla w|^{\beta_{1}}\right)^{\frac{\nu_0}{\beta_{1}}}+C\left(\fint_{Q_{R}^\lambda(x_0,t_0)}|\lambda\nabla w|^{\beta_{1}}\right)^{\frac{\nu_1}{\beta_{1}}},
 	\end{equation}
 provided $\nabla w\in L^{\beta_1}(Q_{R}^\lambda(x_0,t_0))$. Here $C$ is a constant independent of $\beta_1$.
 	% 	where
 	% 	\begin{align}\label{nu01}
 	% 	\nu_0=\prod_{i=1}^{\infty}\frac{\theta\frac{(n-2)(2-p)}{2}+\theta^{i+1}\left(\beta_{1}-\frac{n(2-p)}{2}\right)}{\frac{n(2-p)}{2}+\theta^{i+1}\left(\beta_{1}-\frac{n(2-p)}{2}\right)},~ 	\nu_1=\prod_{i=1}^{\infty}\frac{\frac{\theta n(2-p)}{2}+\theta^{i+1}\left(\beta_{1}-\frac{n(2-p)}{2}\right)}{\frac{n(2-p)}{2}+\theta^{i+1}\left(\beta_{1}-\frac{n(2-p)}{2}\right)},
 	% 	\end{align}
 	% 	and $\theta=\frac{n+2}{n}$.
 \end{lemma}
 \begin{proof} For simplicity, let $(x_0,t_0)=(0,0)$. For any $\varepsilon>0$, let $v$ be the solution of \eqref{eq-v}, then  $v\to w$ in $C^1_{loc}(Q_{2R}^\lambda)$ as $\varepsilon\rightarrow 0$. Let $R/2\leq R_2<R_1\leq R$ and $\psi$ be the one in Lemma \ref{lemw}.
 	Take $v_{x_l}(|\nabla v|^2+\varepsilon^2)^{\frac\mu{2}} \psi^{2}$ as a test function to \eqref{eq-vxl} and integrate in time we obtain
 	\begin{align*}
 		&\sup _{-\lambda^{p-2}R_1^2 \leq t \leq 0} \int_{B_{R_1}}(|\nabla v|^2+\varepsilon^2)^{\frac{\mu+2}{2}} \psi^{2} d x+\int_{Q^{\lambda}_{R_1}}\left|\nabla \left((|\nabla v|^2+\varepsilon^2)^{\frac{\mu+p}{4}}\psi\right)\right|^2  \\
 		&\quad\quad\quad\quad\quad\quad\leq \frac{C}{\lambda^{p-2}R^2} \int_{Q^{\lambda}_{R_1}}(|\nabla v|^2+\varepsilon^2)^{\frac{\mu+2}{2}} \psi +\frac{C}{R^{2}} \int_{Q^{\lambda}_{R_1}}(|\nabla v|^2+\varepsilon^2)^{\frac{\mu+p}{2}}.
 	\end{align*}
 Denote $\theta=\frac{n+2}{n}$,	we mimic the proof  of \cite[Theorem 4]{Choe1} in order to obtain that for any $\mu>p-2$
 	\begin{align*}
 		\fint_{Q_{R_2}^\lambda} (|\nabla v|^2+\varepsilon^2)^\frac{\theta \mu + p+\frac{4}{n}}{2}
 		\leq  \frac{C\lambda^{\frac{2(p-2)}{n}}}{(1-R_2/R_1)^{2\theta}}  \left( \lambda^{2-p} \fint_{Q_{R_1}^\lambda} (|\nabla v|^2+\varepsilon^2)^\frac{\mu+2}{2}    +  \fint_{Q_{R_1}^\lambda}  (|\nabla v|^2+\varepsilon^2)^\frac{\mu+p}{2}    \right)^\theta
 	\end{align*}
 	for any $R_2<R_1<8R_2$. Letting $\varepsilon\to 0$ gives
 	\begin{align}\label{1.2}
 		\fint_{Q_{R_2}^\lambda} |\nabla w|^{\theta \mu + p+\frac{4}{n}}
 		\leq  \frac{C\lambda^{\frac{2(p-2)}{n}}}{(1-R_2/R_1)^{2\theta}}  \left( \lambda^{2-p} \fint_{Q_{R_1}^\lambda} |\nabla w|^{\mu+2}   +  \fint_{Q_{R_1}^\lambda}  |\nabla w|^{\mu+p}    \right)^\theta.
 	\end{align}
 	Put  $V=|\nabla w|$.
 	Thus, for $\widetilde{V}=\lambda V$ we have
 	\begin{align*}
 		\fint_{Q_{R_2}^\lambda} \widetilde{V}^{\theta \mu + p+\frac{4}{n}}
 		\leq  \frac{C}{(1-R_2/R_1)^{2\theta}}  \left(  \fint_{Q_{R_1}^\lambda}  \widetilde{V}^{\mu+2}   +  \fint_{Q_{R_1}^\lambda}  \widetilde{V}^{\mu+p}   \right)^\theta, ~~\forall \mu>p-2.
 	\end{align*}
 	For any  $k\geq 1$, let us set $$r_k=(1+2^{-k+1})\frac{R}{2},~~
 	\beta_{k+1}=\theta\beta_{k}+p-2,~~ \beta_1>\frac{n(2-p)}{2}.
 	$$
 	Note that $\beta_1>\frac{n(2-p)}{2}$, hence the sequence $\{\beta_k\}_{k\geq 1}$ is  increasing. In fact
 	\begin{equation*}
 		\beta_{k}=\frac{n(2-p)}{2}+\left(\frac{n+2}{n}\right)^{k}\left(\beta_1-\frac{n(2-p)}{2}\right)>0.
 	\end{equation*}
 	Apply \eqref{1.2} with $\mu=\beta_k-2$,  $R_1=r_k$, $R_2=r_{k+1}$, for $k\geq 1$.  Then, we obtain
 	\begin{align*}
 		\fint_{Q_{r_{k+1}}^\lambda}\widetilde{V}^{\beta_{k+1}}\leq C2^{2\theta (k+1)} \left(\fint_{Q_{r_{k}}^\lambda}\widetilde{V}^{\beta_{k}}+\left[\fint_{Q_{r_k}^\lambda}\widetilde{V}^{\beta_{k}}\right]^{\frac{\beta_{k}+p-2}{\beta_{k}}}\right)^{\theta}.
 	\end{align*}
 	Therefore, there exists a constant $C=C(n)\geq 2$, such that
 	\begin{equation}\label{1.3}
 		B(k+1)\leq C^{\frac{k}{\beta_{k+1}}} \left(B(k)^{\beta_{k}}+B(k)^{\beta_{k}+p-2}\right)^{\frac{\theta}{ \beta_{k+1} }},
 	\end{equation}
 	with
 	$B(k)=\left(\fint_{Q_{r_{k}}^\lambda}\tilde{V}^{\beta_{k}}\right)^{\frac{1}{\beta_{k}}}$. In the following, we discuss the cases $B(1)\leq 1$ and $B(1)> 1$.
 	\\
 	
 	If  $B(1)\leq 1$,  then
 	we show  that
 	\begin{equation}\label{Z1}
 		B(k+1)\leq C^{b_{k+1}} B(1)^{\prod_{i=1}^{k}\frac{\theta(\beta_{i}+p-2)}{\beta_{i+1}}}  ,
 	\end{equation}
 	where
 	$$b_{k+1}=\frac{k}{\beta_{k+1}}+\frac{\theta}{\beta_{k+1}}+\left(1+\frac{\theta}{\beta_{k+1}}\right)b_{k},~~ b_1=0  .$$
 	We prove \eqref{Z1} by induction.
 	Indeed, it follows from \eqref{1.3}  and the assumptions $B(1)\leq 1$, $  p<2$  that
 	\begin{equation*}
 		B(2)\leq C^{\frac{1}{\beta_{2}}} \left(2B(1)^{\beta_{1}+p-2}\right)^{\frac{\theta}{\beta_{2}}}\leq C^{\frac{1}{\beta_{2}}+\frac{\theta}{\beta_{2}}}B(1)^{\frac{\theta(\beta_{1}+p-2)}{\beta_{2}}}.
 	\end{equation*}
 Hence \eqref{Z1} is true for $k=2$.  Assume that \eqref{Z1} holds for $k=m$.  Since $B(m+1)\leq C^{b_{m+1}}$,  and $\frac{\theta(\beta_{m+1}+p-2)}{\beta_{m+2}}\leq 1$,  we have from \eqref{1.3} that
 	\begin{align*}
 		B(m+2)&\leq C^{\frac{m+1}{\beta_{m+2}}} \left(B(m+1)^{\beta_{m+1}}+B(m+1)^{\beta_{m+1}+p-2}\right)^{\frac{\theta}{\beta_{m+2}}}
 		\\& \leq C^{\frac{m+1}{\beta_{m+2}}} \left(C^{1+b_{m+1}} B(m+1)^{\beta_{m+1}+p-2}\right)^{\frac{\theta}{\beta_{m+2}}}\\& =C^{\frac{m+1}{\beta_{m+2}}+\frac{\theta}{\beta_{m+2}}+\frac{\theta b_{m+1}}{\beta_{m+2}}} B(m+1)^{\frac{\theta(\beta_{m+1}+p-2)}{\beta_{m+2}}} \\&\leq  C^{b_{m+2}} B(1)^{\prod_{i=1}^{m+1}\frac{\theta(\beta_{i}+p-2)}{\beta_{i+1}}} .
 	\end{align*} This implies that \eqref{Z1} also holds for $k=m+1$. Hence \eqref{Z1} holds for any $k\geq 2$.
 	
 	By passing to the limit as $k\rightarrow \infty$ in \eqref{Z1}, we get
 	\begin{equation}\label{1.4a}
 		\|\widetilde{V}\|_{L^\infty(Q_{R/2}^\lambda)}=B(\infty)\leq C^{\lim\limits_{k\to \infty}b_{k}} B(1)^{\prod_{i=1}^{\infty}\frac{\theta(\beta_{i}+p-2)}{\beta_{i+1}}} = C'B(1)^{\nu_0} .
 	\end{equation}
 	Here $\nu_0=\prod_{i=1}^{\infty}\frac{\theta\frac{(n-2)(2-p)}{2}+\theta^{i+1}\left(\beta_{1}-\frac{n(2-p)}{2}\right)}{\frac{n(2-p)}{2}+\theta^{i+1}\left(\beta_{1}-\frac{n(2-p)}{2}\right)}$ and  $C'$ is independent of $\beta_1$.

 	If   $B(1)> 1$, and   our argument above  can be adapted to obtain
 	\begin{equation*}
 		B(k+1)\leq C^{\tilde\mu_{k+1}} B(1)^{\prod_{i=1}^{k}\frac{\theta\beta_{i}}{\beta_{i+1}}},
 	\end{equation*}
 	where
 	$$\tilde\mu_{k+1}=\frac{k}{\beta_{k+1}}+\frac{\theta}{\beta_{k+1}}+\frac{\theta}{\beta_{k+1}}\tilde\mu_{k}, ~~\tilde\mu_1=0 .$$
 Then there is a constant $C>0$ such that
 	\begin{equation}\label{1.4}
 		\|\widetilde{V}\|_{L^\infty(Q_{R/2}^\lambda)}\leq  C B(1)^{\nu_1},
 	\end{equation}
 	where
 	$
 	\nu_1=\prod_{i=1}^{\infty}\frac{\frac{\theta n(2-p)}{2}+\theta^{i+1}\left(\beta_{1}-\frac{n(2-p)}{2}\right)}{\frac{n(2-p)}{2}+\theta^{i+1}\left(\beta_{1}-\frac{n(2-p)}{2}\right)}.
 	$
 	A combination of \eqref{1.4a} and \eqref{1.4} implies \eqref{ZA}.
 \end{proof}
 \begin{remark}
	We have
	\begin{align*}
		\exp\left\{-2(2-p)\left(\beta_1-\frac{n(2-p)}{2}\right)^{-1}\right\} \leq \nu_0\leq\nu_1\leq \exp\left\{\frac{n(2-p)}{2}\left(\beta_1-\frac{n(2-p)}{2}\right)^{-1}\right\}.
	\end{align*}
We also note that $p=2$ iff $\nu_0=\nu_1=1$.
\end{remark}
\begin{proposition}\label{prop1}
	Let $u$ be a weak solution of \eqref{e1}. Then for any $\eta\in(0,1), \kappa_0\in (0,1)$, there exists $\mathbf{h}\in  L^\infty(Q_{R/2}^\lambda(x_0,t_0))$  and $\frac{n(2-p)}{2}<\sigma_1<n+p$ such that
	\begin{align}\label{Linftyofw}
		\lambda\|\mathbf{h}\|_{L^\infty(Q_{R/2}^\lambda(x_0,t_0))}\leq&
		C G_{\gamma_0}\left(\left(\fint_{Q_{2R}^\lambda(x_0,t_0)}|\lambda\nabla u|^{\sigma_1} \right)^{\frac{1}{\sigma_1}}\right)+ C(\eta,\kappa_0)G_{\gamma_0}\left(\left(\fint_{Q_{2R}^\lambda(x_0,t_0)}|\lambda F|^{\sigma_1} \right)^{\frac{1}{\sigma_1}}\right)\\
		&\quad\quad+\eta G_{\gamma_0}\left(\left(\fint_{Q_{2R}^\lambda(x_0,t_0)}|\lambda\nabla u|^{\sigma_1+\kappa_0} \right)^{\frac{1}{\sigma_1+\kappa_0}}\right),\nonumber
	\end{align}
	and
	\begin{align*}
		\left(\fint_{Q_{R/2}^\lambda(x_0,t_0)}|\lambda(\nabla u-\mathbf{h})|^{\sigma_1} \right)^{\frac{1}{\sigma_1}}&\leq \eta G_{\gamma_0}\left(\left(\fint_{Q_{2R}^\lambda(x_0,t_0)}|\lambda\nabla u|^{\sigma_1+\kappa_0} \right)^{\frac{1}{\sigma_1+\kappa_0}}\right)\\&\qquad + C(\eta,\kappa_0)G_{\gamma_0}\left(\left(\fint_{Q_{2R}^\lambda(x_0,t_0)}|\lambda F|^{\sigma_1} \right)^{\frac{1}{\sigma_1}}\right),
	\end{align*}
with $G_{\gamma_0}$ defined in \eqref{formularG}
	for some $\gamma_0\in (0,1)$, provided $[a]_{2R}\leq \eta^{\frac{n^4}{\kappa_0^4(p-1)^4}} $, $\nabla u\in L^{\sigma_1+\kappa_0}(Q_{2R}^\lambda(x_0,t_0))$ and $F\in L^{\sigma_1}(Q_{2R}^\lambda(x_0,t_0))$.
\end{proposition}
\begin{proof}
	For simplicity, we assume that $(x_0,t_0)=(0,0)$. Let $\mathbf{h}=\nabla w$ where $w$ is the weak solution to \eqref{e2}. Then it follows from \eqref{ZA} in Lemma \ref{z} and \eqref{Dwbeta} in Lemma \ref{lemw} that
	\begin{align}
		\lambda\|\mathbf{h}\|_{L^\infty(Q_{R/2}^\lambda)}&\lesssim \left(\fint_{Q_{R}^\lambda}|\lambda\nabla w|^{\beta_{1}}\right)^{\frac{\nu_0}{\beta_{1}}}+\left(\fint_{Q_{R}^\lambda}|\lambda\nabla w|^{\beta_{1}}\right)^{\frac{\nu_1}{\beta_{1}}}\nonumber\\
		&\lesssim A^{\frac{\nu_0}{\beta_1}}+A^{\frac{2\nu_1}{p}}+
		\left(G_{\frac{2p-2}{p}}\left(\fint_{Q^\lambda_{2R}}|\lambda\nabla w|^{p}\right)\right)^{\frac{\nu_0}{\beta_1}}+\left(G_{\frac{2p-2}{p}}\left(\fint_{Q^\lambda_{2R}}|\lambda\nabla w|^{p}\right)\right)^{\frac{\nu_1}{\beta_1}}\label{DwLinfty},
	\end{align}
	where $\beta_{1}>\frac{n(2-p)}{2}$ and $\nu_0$, $\nu_1$ are defined in Lemma \ref{z} and $A$ is defined by \eqref{def-A}.\\
	By \eqref{lem1-case1} in Lemma \ref{u-w}, for $q_0=\frac{2\beta_1}{p}>\frac{n(2-p)}{p}$, $A$ can be estimated by
	\begin{align*}
		A
		\leq &C\left[\fint_{-\lambda^{p-2}
			(2R)^2}^0\left(\fint_{B_{2R}}
		\left(\frac{\lambda}{R}|u-w|\right)^{q_0}dx\right)^\frac{\sigma_2}{q_0}dt\right]
		^\frac{1}{\sigma_2}\\
		&\quad\quad\quad\quad+C\left[\fint_{-\lambda^{p-2}
			(2R)^2}^0\left(\fint_{B_{2R}}
		\left(\frac{\lambda}{R}|u-u_{2R}|\right)^{q_0}dx\right)^\frac{\sigma_2}{q_0}dt\right]
		^\frac{1}{\sigma_2}\\
		\leq &	C(\eta,\kappa_0)\left(\fint_{Q_{2R}^\lambda}|\lambda F|^{\sigma_1}\right)^{\frac{n+p}{\sigma_1(n+2)}}+
		\eta^\frac{\beta_1}{\nu_0}\left(\fint_{Q_{2R}^\lambda}|\lambda \nabla u|^{\sigma_1+\kappa_0}\right)^{\frac{n+p}{(\sigma_1+\kappa_0)(n+2)}}\\
		&\quad\quad\quad\quad+C\left[\fint_{-\lambda^{p-2}
			(2R)^2}^0\left(\fint_{B_{2R}}
		\left(\frac{\lambda}{R}|u-u_{2R}|\right)^{q_0}dx\right)
		^\frac{\sigma_2}{q_0}dt\right]^\frac{1}{\sigma_2},
	\end{align*}
	where
	\[\sigma_1=\frac{q_0(n+p)}{q_0+n+2}> \frac{n(2-p)}{2},\quad\quad\quad \sigma_2=\min\{\sigma_1,2\}.\]
	By Poincar\'e--Sobolev's inequality (see \cite{DiBenedetto1}),
	\begin{align*}
		\left[\fint_{-\lambda^{p-2}
			(2R)^2}^0\left(\fint_{B_{2R}}
		\left(\frac{\lambda}{R}|u-u_{2R}|\right)^{q_0}dx\right)
		^\frac{\sigma_2}{q_0}dt\right]^\frac{1}{\sigma_2}&\lesssim \left[\fint_{-\lambda^{p-2}
			(2R)^2}^0\left(\fint_{B_{2R}}
		|\lambda\nabla u|^{\frac{nq_0}{n+q_0}}dx\right)
		^\frac{\sigma_2(n+q_0)}{nq_0}dt\right]^\frac{1}{\sigma_2}\\
		&\lesssim \left(\fint_{Q^\lambda_{2R}}
		|\lambda\nabla u|^{\sigma_1}dxdt\right)^\frac{1}{\sigma_1},
	\end{align*}
	where we also used the fact that $\frac{nq_0}{n+q_0}\leq \sigma_1$ and $\sigma_2\leq \sigma_1$.
%	Let us denote $\sigma_2=\max\{\sigma_1,2\}$, then
%	\begin{align*}
%		\left[\fint_{-\lambda^{p-2}
%			(2R)^2}^0\left(\fint_{B_{2R}}
%		\left(\frac{\lambda}{R}|u-u_{2R}|\right)^{q_0}dx\right)
%		^\frac{2}{q_0}dt\right]^\frac{1}{2}
%		\lesssim \left(\fint_{Q_{2R}^\lambda}|\lambda \nabla u|^{\sigma_2}\right)
%		^{\frac{1}{\sigma_2}}.
%	\end{align*}
	Hence
	\begin{equation}\label{es-A}
		A\leq	C(\eta,\kappa_0)\left(\fint_{Q_{2R}^\lambda}|\lambda F|^{\sigma_1}\right)^{\frac{n+p}{\sigma_1(n+2)}}+
		\eta^\frac{\beta_1}{\nu_0}\left(\fint_{Q_{2R}^\lambda}|\lambda \nabla u|^{\sigma_1+\kappa_0}\right)^{\frac{n+p}{(\sigma_1+\kappa_0)(n+2)}}+C\left(\fint_{Q_{2R}^\lambda}|\lambda \nabla u|^{\sigma_1}\right)
		^{\frac{1}{\sigma_1}}.
	\end{equation}
	Meanwhile, applying \eqref{lem1-case3} one has
	\begin{align}\label{es-Yw}
		\left(\fint_{Q_{2R}^\lambda} |\lambda\nabla w|^{p}\right)^{\frac{1}{p}}&\leq C\left(\fint_{Q_{2R}^\lambda} |\lambda\nabla (u-w)|^{p}\right)^{\frac{1}{p}}+C\left(\fint_{Q_{2R}^\lambda} |\lambda\nabla u|^{p}\right)^{\frac{1}{p}}\\
		&\leq C(\eta)\fint_{Q_{2R}^\lambda}  |\lambda F|^{p}+\eta^\frac{\beta_1}{\nu_0}\left(\fint_{Q_{2R}^\lambda}|\lambda\nabla u|^{\sigma_1}\right)^\frac{p}{\sigma_1}+C\left(\fint_{Q_{2R}^\lambda}  |\lambda \nabla u|^{p}\right)^\frac{1}{p}.\nonumber
	\end{align}
	Plugging \eqref{es-A} and \eqref{es-Yw} in \eqref{DwLinfty} gives \eqref{Linftyofw}.
	%\begin{align*}
	%  &\lambda\|\nabla w\|_{L^\infty(Q_{R/2}^\lambda)}\lesssim\sup_{\gamma\in [\delta,1/\delta]}(X(\nabla u)+X(F))^{\gamma}.
	%\end{align*}
	%for some $\delta\in (0,1)$.

	Next, using interpolation inequality and Young's inequality, one has
	\begin{align*}
		\left(\fint_{Q_{R/2}^\lambda}|\lambda(\nabla u-\mathbf{h})|^{\sigma_1} \right)^{\frac{1}{\sigma_1}}
		&\leq C(\eta,\kappa_0)\left(\fint_{Q_{R/2}^\lambda}|\lambda\nabla(u-w)|^{p} \right)^{\frac{1}{p}}+\eta^2\left(\fint_{Q_{R/2}^\lambda}|\lambda\nabla u|^{\sigma_1+\kappa_0} \right)^{\frac{1}{\sigma_1+\kappa_0}}\\&\quad\quad\quad\quad+\eta^2\left(\fint_{Q_{R/2}^\lambda}|\lambda\nabla w|^{\sigma_1+\kappa_0} \right)^{\frac{1}{\sigma_1+\kappa_0}}.
	\end{align*}
	By \eqref{lem1-case3}, \eqref{Linftyofw} and H\"{o}lder's inequality, we have
	\begin{align*}
		&\left(\fint_{Q_{R/2}^\lambda}|\lambda(\nabla u-\mathbf{h})|^{\sigma_1} \right)^{\frac{1}{\sigma_1}}\\
		&\leq C(\eta,\kappa_0) \left(\fint_{Q_{R/2}^\lambda}  |\lambda F|^{p}\right)^\frac{1}{p} +2\eta^2\left(\fint_{Q_{R/2}^\lambda}|\lambda\nabla u|^{\sigma_1+\kappa_0} \right)^{\frac{1}{\sigma_1+\kappa_0}}+\eta^2 \lambda\|\mathbf{h}\|_{L^\infty(Q_{R/2}^\lambda)}   \\
		&\leq C(\eta,\kappa_0)G_{\gamma_0}\left(\left(\fint_{Q_{2R}^\lambda}|\lambda F|^{\sigma_1} \right)^{\frac{1}{\sigma_1}}\right)+\eta G_{\gamma_0}\left(\left(\fint_{Q_{2R}^\lambda}|\lambda\nabla u|^{\sigma_1+\kappa_0} \right)^{\frac{1}{\sigma_1+\kappa_0}}\right).
	\end{align*}
	Then we finish the proof.
\end{proof}
\begin{remark}\label{repr1}
	Note that if $p>\frac{2n}{n+2}$, we do not need Lemma \ref{lemw}. Applying Lemma \ref{z} with $\beta_1=p$, we obtain the $L^\infty$ estimate of $\nabla w$. Then Proposition \ref{prop1} holds with $\sigma_1=p$ and $\kappa_0=0$.
\end{remark}
\subsection{Boundary estimates}

In this subsection, we consider the corresponding estimates near the boundary. Let $\Omega$ be a bounded domain in $\mathbb{R}^n$ with $C^2$ boundary and let $x_0\in \partial \Omega$. For $x\in\mathbb{R}^n$, we write $x=(x',x_n)$ where $x'\in \mathbb{R}^{n-1}$ and $x_n\in \mathbb{R}$. Since $\partial\Omega$ is $ C^2$, there exist $\bar R>0$ and a $C^2$ function $\mathfrak{g}:\mathbb{R}^{n-1}\to \mathbb{R}$ such that
\[\Omega \cap B_{\bar R}(x_0)=\{x=(x',x_n)\in B_{\bar R}(x_0): x_n>\mathfrak{g}(x')\},\]
and
\begin{equation}\label{estg}
	|\nabla\mathfrak{g}(x')|\lesssim |x'-x_0'|,\quad 	|\nabla^2\mathfrak{g}(x')|\lesssim 1,
\end{equation}
for any $|x'-x_0'|<\bar R$. With $u$ being a weak solution to \eqref{e1}, we consider the unique solution
\begin{equation*}
	w\in L^p((-\lambda^{p-2}(4 R)^2+t_0,t_0), W_{0}^{1,p}(\Omega_{4 R}(x_0)))+u
\end{equation*}
to the following equation
\begin{equation}\label{eqwbdry}
	\left\{ \begin{array}{ll}
		w_t -\operatorname{div}((a(t))_{\Omega_{4 R}(x_0)}|\nabla w|^{p-2}\nabla w)= 0\ & \text{in}\ (-\lambda^{p-2}(4R)^2+t_0,t_0)\times \Omega_{4R}(x_0), \\
		w = u \ & \text{on}\ (-\lambda^{p-2} (4R)^2+t_0,t_0)\times \partial \Omega_{4 R}(x_0), \\
	\end{array} \right.
\end{equation}
where we define $\Omega_{4 R}(x_0)=\Omega\cap B_{4 R}(x_0)$, and $0<R< R_0$ with $R_0$ will be fixed later in our proof. We have the following boundary counterpart of Lemma \ref{lemw}.
\begin{lemma}\label{lemBdry1}
  Let $w$ be a weak solution to equation \eqref{eqwbdry}. There exists $R_0>0$ small depending on $\partial \Omega$ such that for any $(x_0,t_0)\in \partial\Omega\times [0,T]$, any $0<R< R_0$,
   and any $\beta_1>\frac{n(2-p)}{2}$, we have
 \begin{align}\label{es-dwbebdry}
	\fint_{{{Q}}^\lambda_{R/2}(x_0,t_0)}(\lambda |\nabla  {w}|)^{\beta_1}\
	\leq C
	{A}_1
	^\frac{2\beta_1}{ p}+C	{A}_1+CG_{\frac{2p-2}{p}}\left(\fint_{{Q}^\lambda_{4R}(x_0,t_0)}\left(\lambda|\nabla{w}|\right)^{p}\right),
	\end{align}
where
\begin{equation}\label{def-A1}
{A}_1=\left[\fint_{t_0-\lambda^{p-2}
		(4R)^2}^{t_0}\left(\fint_{B_{4R}(x_0)\cap \Omega} \left(\frac{\lambda}{R}|{w}|\right)^{q_0}dx\right)^\frac{\sigma_2}{q_0}dt\right]
	^\frac{1}{\sigma_2},
\end{equation}
with $q_0=\frac{2\beta_1}{p}>\frac{n(2-p)}{p}$ and $\sigma_2=\min\{\frac{2\beta_1(n+p)}{2\beta_1+(n+2)p},2\}$. And the function $G_{\frac{2p-2}{p}}$ is defined in \eqref{formularG}.
\end{lemma}
\begin{proof}
Without loss of generality, we assume that $x_0=0\in\partial\Omega$ and $t_0=0$.
Now let $\mathcal{T}:\mathbb{R}^n\to\mathbb{R}^n$ be a $C^2$-diffeomorphism defined by
\[y=\mathcal{T}(x)=\mathcal{T}((x',x_n))=(x',x_n-\mathfrak{g}(x'))\]
and also $\mathcal{T}(0)=0$. Set
\begin{equation*}
  x=\mathcal{T}^{-1}(y).
\end{equation*}
Next, we can choose $M$ large enough such that for any $0<R<R_0:=\bar R/M$, there holds 
\begin{equation}\label{subset}
B_{R/2}\cap \Omega\subset	\mathcal{T}^{-1}(B^+_{R})\subset	\mathcal{T}^{-1}(B^+_{3R})\subset B_{4R}\cap \Omega \subset B_{\bar R}\cap \Omega,
\end{equation}
where $B^+_{3R}=B_{3R}\cap \mathbb{R}^n_+$. 
Define \[\widetilde{w}(t,y)=w(t,
\mathcal{T}^{-1}(y))\]
for all $y\in B^+_{3R}$.
After a change of variables on \eqref{eqwbdry}, we see that  $\widetilde{w}$ is a solution to
\begin{equation*}
	\left\{ \begin{array}{ll}
		\widetilde{w}_t -\operatorname{div}_y\left((a(t))_{\Omega_{4 R}}|D \mathcal{T}\nabla_y \widetilde{w}|^{p-2}(D \mathcal{T})^T D \mathcal{T}\nabla_y \widetilde{w}\right)= 0& \text{in}~~(-\lambda^{p-2}(3R)^2,0)\times B^+_{3R}, \\
		\widetilde{w} = 0& \text{on}~~(-\lambda^{p-2}(3R)^2,0)\times\left[B_{3R}\cap\partial\mathbb{R}^n_+\right]. \\
	\end{array} \right.
\end{equation*}
By the definition of $\mathcal{T}$, it is easy to check that
\begin{align*}
	D\mathcal{T}=\left(
	\begin{array}{cc}
		I_{n-1}   & -\nabla \mathfrak{g} \\
		0 & 1
	\end{array}\right),\quad\quad\quad \quad (D\mathcal{T})^{-1}=\left(
	\begin{array}{cc}
		I_{n-1}   & \nabla \mathfrak{g} \\
		0 & 1
	\end{array}\right).
\end{align*}
Hence
\begin{align}\label{DT}
	\operatorname{det} (D \mathcal{T})=\operatorname{det} (D \mathcal{T})^{-1}=1,\quad \text{and}\quad\|D \mathcal{T}-I\|_{L^\infty}+\|(D \mathcal{T})^{-1}-I\|_{L^\infty}\leq CR.\end{align}
Since $\widetilde{w}=0$ on $B_{3R}(0)\cap\partial\mathbb{R}^n_+$, we make an odd extension of $\widetilde{w}$ in $B_{3R}(0)$ and denote it as $\overline{w}$, i.e. for $|y|\leq 3R$, define
\begin{equation*}
	\overline{w}(y)=
	\begin{cases}
		\widetilde{w}(y),\qquad\qquad~ \text{if}\ y_n\geq 0;\\
		-\widetilde{w}(y',-y_n),~~~ \text{if} \ y_n<0.
	\end{cases}
\end{equation*}
Then $\overline{w}$ satisfies
\begin{align}\label{eqbarw}
	\overline{w}_t -\operatorname{div}_y\left((a(t))_{\Omega_{4 R}}\mathbf{A}_0(y,\nabla_y\overline{w})\right)= 0\quad \text{in}\ Q^\lambda_{3R},
\end{align}
where \[\mathbf{A}_0( y,\xi)=|D \mathcal{T} \xi|^{p-2}(D \mathcal{T})^TD \mathcal{T}\xi\quad\quad \text{for}\ \xi\in \mathbb{R}^n.\]

Like the proof of Lemma \ref{lemw}, we will approximate $\overline{w}$ by $	\overline{v}$. For any $\varepsilon>0$, let $	\overline{v}$ be a solution of 
\begin{equation}\label{eqbarv}
\left\{	\begin{array}{ll}
	\overline{v}_t -\operatorname{div}_y\left((a(t))_{\Omega_{4 R}}\mathbf{A}(y,\nabla_y\overline{v})\right)= 0& \text{in}\quad Q^\lambda_{3R},\\
	\overline{v}=\overline{w} &\text{on} \quad\partial_p Q^\lambda_{3R},
	\end{array}\right.
\end{equation}
where \[\mathbf{A}( y,\xi)=(|D \mathcal{T} \xi|^2+\varepsilon^2)^\frac{p-2}{2}(D \mathcal{T})^TD \mathcal{T}\xi.\]
By the classical regularity theory,  $\overline{v}$ is smooth for any $\varepsilon>0$. Moreover, $\overline{v}\to \overline{w}$ in $C^1_{loc}(Q_{3R}^\lambda)$.
$	\mathbf{A}(0,\xi)=(|\xi|^2+\varepsilon^2)^\frac{p-2}{2}\xi$ and $	\partial_\xi\mathbf{A}(0,\xi)=\left(Id+(p-2)\frac{\xi\otimes \xi}{|\xi|^2+\varepsilon^2}\right)(|\xi|^2+\varepsilon^2)^\frac{p-2}{2}$. Moreover, by \eqref{estg} it is easy to check that
\begin{equation*}
	|\partial_\xi\mathbf{A}(y,\xi)-\partial_\xi\mathbf{A}(0,\xi)|\leq c|y| (|\xi|^2+\varepsilon^2)^\frac{p-2}{2}.
\end{equation*}
Hence for any $|y|\leq 3R \leq \frac{p-1}{2c}$, one has
\begin{equation}\label{DA}
	\partial_\xi\mathbf{A}(y,\xi)\geq \frac{p-1}{2}(|\xi|^2+\varepsilon^2)^\frac{p-2}{2}Id.
\end{equation}
By \eqref{estg} we obtain 
\begin{align}\label{dyA}
	|\partial_y \mathbf{A}(y,\xi)|\lesssim (|\xi|^2+\varepsilon^2)^\frac{p-1}{2}.
\end{align}

%\begin{equation}\label{eq-v1}
%	\left\{ \begin{array}{ll}
%		v_t -\operatorname{div}((a(t))_{B_{2R}}(|\nabla v|^2+\varepsilon^2)^{\frac{p-2}{2}}\nabla v)= 0 & \text{in}\ (-\lambda^{p-2}\bar R^2,0)\times \Omega_{4 R}, \\
%		v = w & \text{on}\ (-\lambda^{p-2} (4R)^2,0)\times \partial \Omega_{4R},
%	\end{array} \right.
%\end{equation}
%then $v\to w$ in $C_{loc}^1((-\lambda^{p-2}(4 R)^2,0)\times \Omega_{4 R})$.
%Define \[\widetilde{v}(t,y)=v(t,
%\mathcal{T}^{-1}(y))\]
%for all $y\in B^+_{2R}$.
%After a change of variables, we see that  $\widetilde{v}$ is a solution to
%\begin{equation*}
%	\left\{ \begin{array}{ll}
%		\widetilde{v}_t -\operatorname{div}_y\left((a(t))_{B_{2R}}(|D \mathcal{T}\nabla_y \widetilde{v}|^2+\varepsilon^2)^\frac{p-2}{2}(D \mathcal{T})^T D \mathcal{T}\nabla_y \widetilde{v}\right)= 0&\text{in}\ (-\lambda^{p-2}(2R)^2,0)\times B^+_{2R}, \\
%		\widetilde{v} = 0 &\text{on}\ (-\lambda^{p-2}(2R)^2,0)\times\left[B_{2R}\cap\partial\mathbb{R}^n_+\right]. \\
%	\end{array} \right.
%\end{equation*}

We now estimate $\overline{v}$. Let $R\leq R_2<R_1\leq 2R$, the cutoff function $\psi$, and parameters $k, k', \mu, q_1$ be as defined in Lemma \ref{lemw}. One has 
\begin{equation*}
	\begin{split}
		&\int_{Q^{\lambda}_{R_1}}
		(|\nabla_y \overline{v}|^2+\varepsilon^2)^{\frac{\mu}{2}}|\nabla_y \overline{v}|^2\psi^{k} \\&\lesssim  \int_{-\lambda^{p-2}R_1^2}^{0} \left(\int _{B_{R_1}} (|\nabla_y \overline{v}|^2+\varepsilon^2)^{\frac{\mu}{2}}|\nabla_y \overline{v}|^2\psi^{k+1}dx\right)^{1-\theta_0}  \left(\int_{B_{R_1}} (|\nabla_y \overline{v}|^2+\varepsilon^2)^{\frac{\mu}{2}}|\nabla_y \overline{v}|^2\psi^{k'} dy\right)^{\theta_0} dt\\
		&\lesssim\sup _{-\lambda^{p-2}R_1^2 \leq t \leq 0} \left(\int_{B_{R_1}}(|\nabla_y\overline{v} |^2+\varepsilon^2)^{\frac{\mu+2}{2}} \psi^{k+1} d y\right)^{1-\theta_0}\int_{-\lambda^{p-2}R_1^2}^{0}\left(\int_{B_{R_1}} (|\nabla_y \overline{v}|^2+\varepsilon^2)^{\frac{\mu}{2}}|\nabla_y\overline{v}|^2\psi^{k'} dy\right)^{\theta_0} dt.
	\end{split}
\end{equation*}
 Using integration by parts and applying H\"older's inequality we obtain
\begin{equation}\nonumber
	\begin{split}
		\int_{B_{R_1}}(|\nabla_y \overline{v}|^2+\varepsilon^2)^{\frac{\mu}{2}}|\nabla_y \overline{v}|^{2} \psi^{k'} d y&=-\int_{B_{R_1}} \overline{v} \operatorname{div}_y\left((|\nabla_y \overline{v}|^2+\varepsilon^2)^{\frac{\mu}{2}} \nabla_y \overline{v} \psi^{k'}\right) d y \\
		&\lesssim \left(\int_{B_{R_1}}|\overline{v}|^{q_1}dy\right)^\frac{1}{q_1}\left(\int_{B_{R_1}}\left((|\nabla_y \overline{v}|^2+\varepsilon^2)^{\frac{\mu}{2}}\left|\nabla_y^{2}\overline{v}\right| \psi^{k'}\right)^{q_1'} d y\right)^\frac{1}{q_1'}\\
		&\quad+\left(\int_{B_{R_1}}|\overline{v}|^{q_1}dy\right)^\frac{1}{q_1} \left(\int_{B_{R_1}}\left((|\nabla_y \overline{v}|^2+\varepsilon^2)^{\frac{\mu+1}{2}}|\nabla_y \psi| \psi^{k'-1}\right)^{q_1'} d y\right)^\frac{1}{q_1'}.
	\end{split}
\end{equation}
where $q_1'$ satisfies $\frac{1}{q_1}+\frac{1}{q_1'}=1$. By H\"older's inequality we also obtain \eqref{term1} and \eqref{term2} with $\nabla v$, $\nabla^2 v$  replaced by $\nabla_y\overline{v}$ and $\nabla_y^2\overline{v}$, respectively. Following the proof of  Lemma \ref{lemw}, we get
\begin{equation}\label{es-dbarv}
	\begin{split}
		&\int_{Q^{\lambda}_{R_1}}(|\nabla_y \overline{v}|^2+\varepsilon^2)^{\frac{\mu}{2}}|\nabla_y \overline{v}|^2\psi^{k} \\
		&\lesssim \left(\int_{-\lambda^{p-2}R_1^2}^{0} \left(\int_{B_{R_1}}|\overline{v}|^{q_1}dy\right)^\frac{\sigma_2}{q_1}dt\right)^{\frac{\theta_0}{\sigma_2}}\left(\int_{Q^{\lambda}_{R_1}}(|\nabla_y \overline{v}|^2+\varepsilon^2)^{\frac{\mu+p-2}{2}}\left|\nabla_y^{2} \overline{v}\right|^{2} \psi^{k+1} \right)^{\frac{\theta_0}{2}}\\ &
		\qquad\qquad\qquad\qquad\times \left(\sup_{ -\lambda^{p-2}R_1^2\leq t\leq 0}\int_{B_{R_1}}(|\nabla_y \overline{v}|^2+\varepsilon^2)^\frac{\mu+2}{2} \psi^{k+1} d y\right)^{\frac{(\mu+2-p)\theta_0}{2(\mu+2)}+1-\theta_0}\\
		&\qquad+\frac{(\lambda^{p-2}R^2)^{\frac{\theta_0}{2}}}{R^{(1-\frac{n}{\mu+2}+\frac{n}{q_1})\theta_0}}\left(\int_{-\lambda^{p-2}R_1^2}^{0} \left(\int_{B_{R_1}}|\overline{v}|^{q_1}dy\right)^\frac{\sigma_2}{q_1}dt\right)^{\frac{\theta_0}{\sigma_2}} \\
		&\qquad\qquad\qquad\qquad\times\left(\sup_{ -\lambda^{p-2}R_1^2\leq t\leq 0}\int_{B_{R_1}}(|\nabla_y \overline{v}|^2+\varepsilon^2)^{\frac{\mu+2}{2}}\psi^{k+1} d y\right)^{\frac{(\mu+1)\theta_0}{\mu+2}+1-\theta_0}.
	\end{split}
\end{equation}
Next, we differentiate \eqref{eqbarv} with respect to $y_l$ to get
%\begin{align*}
%\left(\overline{w}_{y_l}\right)_t-\partial_{y_l}\operatorname{div}_y\left(\mathbf{A}(t,y,\nabla_y\overline{w})\right)= 0
%\end{align*}
\begin{equation}\label{eqDbarv}
	\left(\overline{v}_{y_l}\right)_t-	\operatorname{div}_y\left((a(t))_{\Omega_{4 R}}\partial_\xi\mathbf{A}( y,\xi)|_{\xi=\nabla_y\overline{v}} \nabla_y\overline{v}_{y_l}\right)=\operatorname{div}_y\left((a(t))_{\Omega_{4 R}}\partial_{y_l}\mathbf{A}(y,\xi)\right).
\end{equation}
 Testing \eqref{eqDbarv} by $\overline{v}_{y_l}(|\nabla_y\overline{v}|^2+\varepsilon^2)^\frac{\mu}{2}\psi^{k+1}$, applying \eqref{estg}, \eqref{DA} and \eqref{dyA} we obtain
\begin{align*}
	&\sup _{-\lambda^{p-2}R_1^2 \leq t \leq 0} \int_{B_{R_1}}(|\nabla_y\overline{v}|^2+\varepsilon^2)^\frac{\mu+2}{2} \psi^{k+1} d y+\int_{Q^{\lambda}_{R_1}}(|\nabla_y \overline{v}|^2+\varepsilon^2)^\frac{\mu+p-2}{2}\left|\nabla_y^{2}\overline{v}\right|^{2} \psi^{k+1}  \\
	&\leq \frac{C}{\lambda^{p-2}R^2} \int_{Q^{\lambda}_{R_1}}(|\nabla_y\overline{v}|^2+\varepsilon^2)^\frac{\mu+2}{2} \psi^k +\frac{C}{R^{2}} \int_{Q^{\lambda}_{R_1}}(|\nabla_y\overline{v}|^2+\varepsilon^2)^\frac{\mu+p}{2}  \psi^{k-1} \\
	&\quad\quad+\tilde C\int_{Q^{\lambda}_{R_1}}(|\nabla_y \overline{v}|^2+\varepsilon^2)^\frac{\mu+p-1}{2}\left|\nabla_y^{2}\overline{v}\right| \psi^{k+1}.
\end{align*}
Note that by H\"older's inequality and Young's inequality we have
\begin{align*}
	&\int_{Q^{\lambda}_{R_1}}(|\nabla_y \overline{v}|^2+\varepsilon^2)^\frac{\mu+p-1}{2}\left|\nabla_y^{2}\overline{v}\right| \psi^{k+1} \\
	&\quad\quad\leq \frac{1}{10\tilde C}\int_{Q^{\lambda}_{R_1}}(|\nabla_y \overline{v}|^2+\varepsilon^2)^\frac{\mu+p-2}{2}\left|\nabla_y^{2}\overline{v}\right|^{2} \psi^{k+1} +C \int_{Q^{\lambda}_{R_1}}(|\nabla_y\overline{v}|^2+\varepsilon^2)^\frac{\mu+p}{2}  \psi^{k-1}  \\
	&\quad\quad\leq \frac{1}{10\tilde C}\int_{Q^{\lambda}_{R_1}}(|\nabla_y \overline{v}|^2+\varepsilon^2)^\frac{\mu+p-2}{2}\left|\nabla_y^{2}\overline{v}\right|^{2} \psi^{k+1} +\frac{C}{R^{2}} \int_{Q^{\lambda}_{R_1}}(|\nabla_y\overline{v}|^2+\varepsilon^2)^\frac{\mu+p}{2}  \psi^{k-1}  ,
\end{align*}
where we also used the  fact that $R\leq 1$.
Hence
\begin{align*}
	&\sup _{-\lambda^{p-2}R_1^2 \leq t \leq 0} \int_{B_{R_1}}(|\nabla_y\overline{v}|^2+\varepsilon^2)^\frac{\mu+2}{2} \psi^{k+1} d y+\int_{Q^{\lambda}_{R_1}}(|\nabla_y \overline{v}|^2+\varepsilon^2)^\frac{\mu+p-2}{2}\left|\nabla_y^{2}\overline{v}\right|^{2} \psi^{k+1}  \\
	\leq &\,\frac{C}{\lambda^{p-2}R^2} \int_{Q^{\lambda}_{R_1}}(|\nabla_y\overline{v}|^2+\varepsilon^2)^\frac{\mu+2}{2} \psi^k +\frac{C}{R^{2}} \int_{Q^{\lambda}_{R_1}}(|\nabla_y\overline{v}|^2+\varepsilon^2)^\frac{\mu+p}{2}  \psi^{k-1} .
\end{align*}
Combining this with \eqref{es-dbarv}, we obtain 
	\begin{align*}
	&\int_{Q^{\lambda}_{R_1}}(|\nabla_y v|^2+\varepsilon^2)^{\frac{\mu}{2}}|\nabla_y v|^2\psi^{k}\\& \lesssim \left(\int_{-\lambda^{p-2}R_1^2}^{0} \left(\int_{B_{R_1}}|\overline{v}|^{q_1}dy\right)^\frac{\sigma_2}{q_1}dt\right)^{\frac{\theta_0}{\sigma_2}} \\&\quad\quad\times\left\{\left(\frac{1}{\lambda^{p-2}R^2}\int_{Q^{\lambda}_{R_1}}(|\nabla_y \overline{v}|^2+\varepsilon^2)^{\frac{\mu+2}{2}} \psi^{k}+\frac{1}{R^2}\int_{Q^{\lambda}_{R_1}}(|\nabla_y \overline{v}|^2+\varepsilon^2)^{\frac{\mu+p}{2}} \psi^{k}\right)^{1-\frac{p\theta_0}{2(\mu+2)}}\right.\\
	&\quad\quad\quad+\left.\frac{(\lambda^{p-2}R^2)^{\frac{\theta_0}{2}}}{R^{(1-\frac{n}{\mu+2}+\frac{n}{q_1})\theta_0}}\left(\frac{1}{\lambda^{p-2}R^2}\int_{Q^{\lambda}_{R_1}}(|\nabla_y \overline{v}|^2+\varepsilon^2)^{\frac{\mu+2}{2}} \psi^{k}+\frac{1}{R^2}\int_{Q^{\lambda}_{R_1}}(|\nabla_y \overline{v}|^2+\varepsilon^2)^{\frac{\mu+p}{2}} \psi^{k}\right)^{1-\frac{\theta_0}{\mu+2}}\right\}.
\end{align*}
Following the proof in Lemma \ref{lemw}, let $\varepsilon\rightarrow 0$, then $\overline{v}\to \overline{w}$ in $C^1(Q_{3R/2}^\lambda)$. 
 By Young's inequality  
we obtain 
\begin{align}\nonumber
	\fint_{Q^\lambda_{R_2}}|\lambda\nabla_y \overline{w}|^{\mu+2}
	\lesssim
	\tilde{A}_1
	^\frac{2(\mu+2)}{ p}+\tilde{A}_1+\fint_{Q^\lambda_{R_1}}|\lambda\nabla_y \overline{w}|^{\mu+p},
\end{align}
where
\begin{align*}
	\tilde{A}_1=\left[\fint_{-\lambda^{p-2}(2R)^2}^{0} \left(\fint_{B_{2R}}\left(\frac{\lambda}{R}|\overline{w}|\right)^{q_0}dy\right)^\frac{\sigma_2}{q_0}dt\right]^{\frac{1}{\sigma_2}}.
\end{align*}
We mimic the iteration procedure in Lemma \ref{lemw} to get that for  $\beta_1>\frac{n(2-p)}{2}$,
\begin{align}\label{eq-wbar2}
	\fint_{Q^\lambda_{R}}|\lambda\nabla_y \overline{w}|^{\beta_{1}}
	\lesssim
	\tilde{A}_1
	^\frac{2\beta_1}{ p}+\tilde{A}_1+G_{\frac{2p-2}{p}}\left(\fint_{Q^\lambda_{2R}}|\lambda\nabla_y\overline{w} |^p\right).
\end{align}
Moreover, since $\overline{w}$ is odd with respect to $y_n$ in $B_{2R}$, \eqref{eq-wbar2} also holds in the half parabolic cylinder, i.e.,
\begin{align}\label{eq-wbar3}
	\fint_{Q^{\lambda+}_{R}}|\lambda\nabla_y \overline{w}|^{\beta_{1}}
	\lesssim
	\tilde{A}_2
	^\frac{2\beta_1}{ p}+\tilde{A}_2+G_{\frac{2p-2}{p}}\left(\fint_{Q^{\lambda+}_{2R}}|\lambda\nabla_y\overline{w} |^p\right),
\end{align}
where we denote $Q^{\lambda+}_{R}= B_{R}^+\times(-\lambda^{p-2}R^2,0)$ and
\begin{equation*}
	\tilde{A}_2=\left[\fint_{-\lambda^{p-2}
		(2R)^2}^0\left(\fint_{B_{2 R}^+}
	\left(\frac{\lambda}{R}|\overline{w}|\right)^{q_0}dy\right)^\frac{\sigma_2}{q_0}dt\right]
	^\frac{1}{\sigma_2}.
\end{equation*}
For any $x\in B_{R/2}\cap \Omega$, we have $w(x,t)=\overline{w}(\mathcal{T}x,t)$ and $\nabla w(x,t)=D \mathcal{T}\nabla_y \overline{w}(\mathcal{T}x,t)$. By \eqref{subset}, \eqref{DT} and \eqref{eq-wbar3} we obtain \eqref{es-dwbebdry}.
%have for $\beta_1>\frac{n(2-p)}{2}$,
%\begin{align}\label{eq-vbar4}
%	\fint_{Q^{\lambda}_{R/2}}|\lambda\nabla {w}|^{\beta_{1}}
%	\lesssim
%	{A}_1
%	^\frac{2\beta_1}{ p}+{A}_1^{2}+G_{\frac{2p-2}{p}}\left(\fint_{Q^{\lambda}_{4R}}|\lambda\nabla {w} |^{p}\right),
%\end{align}
%where
%\begin{equation*}
%	{A}_1=\left[\fint_{-\lambda^{p-2}
%		(4R)^2}^0\left(\fint_{B_{4R}\cap\Omega}
%	\left(\frac{\lambda}{R} |w|\right)^{q_0}dy\right)^\frac{2}{q_0}dt\right]
%	^\frac{1}{2}.
%\end{equation*}
%We complete the proof.
\end{proof}

\begin{lemma}\label{lemBdry2}
  Let $w$ be a weak solution to equation \eqref{eqwbdry}. There exists $R_0>0$ small depending on $\partial \Omega$ such that for any $(x_0,t_0)\in \partial\Omega\times[0,T]$, any $0<R< R_0$,
and any $\beta_1>\frac{n(2-p)}{2}$, we have
\begin{equation*}
	\lambda\|\nabla w\|_{L^\infty({Q}^\lambda_{R/2}(x_0,t_0))}\leq C\left(\fint_{{Q}^\lambda_{4R}(x_0,t_0)}|\lambda\nabla w|^{\beta_{1}}\right)^{\frac{\nu_0}{\beta_{1}}}+C\left(\fint_{{Q}^\lambda_{4R}(x_0,t_0)}|\lambda\nabla w|^{\beta_{1}}\right)^{\frac{\nu_1}{\beta_{1}}}
\end{equation*}
for some $\nu_0, \nu_1$ depending on $\beta_1, p, n$.
\end{lemma}
\begin{proof}
	For simplicity, we assume that $(x_0,t_0)=(0,0)$. Following Lemma \ref{lemBdry1}, after transformation and extension, we get $\overline{w}$ satisfying \eqref{eqbarw}. Let $\overline{v}$ be a solution to the approximate system  \eqref{eqbarv}. We have $\overline{v}\to \overline{w}$ in $C^1_{loc}(Q_{3R}^\lambda)$ as $\varepsilon\rightarrow 0$. Let $R\leq R_2<R_1\leq 2R$ and let $\psi$ be the cut-off function defined in Lemma \ref{lemw}.  Take $\overline{v}_{y_l}(|\nabla_y \overline{v}|^2+\varepsilon^2)^{\frac\mu{2}} \psi^{2}$ as a test function to \eqref{eqDbarv} and integrate in time we obtain
	\begin{align*}
		&\sup _{-\lambda^{p-2}R_1^2 \leq t \leq 0} \int_{B_{R_1}}(|\nabla_y \overline{v}|^2+\varepsilon^2)^{\frac{\mu+2}{2}} \psi^{2} d y+\int_{Q^{\lambda}_{R_1}}\left|\nabla_y \left((|\nabla_y \overline{v}|^2+\varepsilon^2)^{\frac{\mu+p}{4}}\psi\right)\right|^2  \\
		&\quad\quad\leq \frac{C}{\lambda^{p-2}R^2} \int_{Q^{\lambda}_{R_1}}(|\nabla_y \overline{v}|^2+\varepsilon^2)^{\frac{\mu+2}{2}} \psi +\frac{C}{R^{2}} \int_{Q^{\lambda}_{R_1}}(|\nabla_y \overline{v}|^2+\varepsilon^2)^{\frac{\mu+p}{2}}\\&\qquad\qquad+\tilde C\int_{Q^{\lambda}_{R_1}}(|\nabla_y \overline{v}|^2+\varepsilon^2)^\frac{\mu+p-1}{2}\left|\nabla_y^{2}\overline{v}\right| \psi^{2} .
	\end{align*}
	Note that by H\"older's inequality and Young's inequality,
	\begin{align*}
		&\int_{Q^{\lambda}_{R_1}}(|\nabla_y \overline{v}|^2+\varepsilon^2)^\frac{\mu+p-1}{2}\left|\nabla_y^{2}\overline{v}\right| \psi^{2}\\
		&\quad\quad\leq \frac{1}{10\tilde C} \int_{Q^{\lambda}_{R_1}}\left|\nabla_y\left((|\nabla_y\overline{v}|^2+\varepsilon^2)^{\frac{\mu+p}{4}}\psi\right)\right|^2+\frac{C}{R^{2}} \int_{Q^{\lambda}_{R_1}}(|\nabla_y \overline{v}|^2+\varepsilon^2)^\frac{\mu+p}{2}.
	\end{align*}
	Hence
	\begin{align*}
		&\sup _{-\lambda^{p-2}R_1^2 \leq t \leq 0} \int_{B_{R_1}}(|\nabla_y \overline{v}|^2+\varepsilon^2)^{\frac{\mu+2}{2}} \psi^{2} d y+\int_{Q^{\lambda}_{R_1}}\left|\nabla_y \left((|\nabla_y \overline{v}|^2+\varepsilon^2)^{\frac{\mu+p}{4}}\psi\right)\right|^2  \\
		&\quad\quad\leq \frac{C}{\lambda^{p-2}R^2} \int_{Q^{\lambda}_{R_1}}(|\nabla_y \overline{v}|^2+\varepsilon^2)^{\frac{\mu+2}{2}} \psi +\frac{C}{R^{2}} \int_{Q^{\lambda}_{R_1}}(|\nabla_y \overline{v}|^2+\varepsilon^2)^{\frac{\mu+p}{2}},
	\end{align*}
	which yields
	\begin{align*}
		\fint_{Q_{R_2}^\lambda} (|\nabla_y \overline{v}|^2+\varepsilon^2)^\frac{\theta \mu + p+\frac{4}{n}}{2}
		\leq  \frac{C\lambda^{\frac{2(p-2)}{n}}}{(1-R_2/R_1)^{2\theta}}  \left( \lambda^{2-p} \fint_{Q_{R_1}^\lambda} (|\nabla_y \overline{v}|^2+\varepsilon^2)^\frac{\mu+2}{2}    +  \fint_{Q_{R_1}^\lambda}  (|\nabla_y \overline{v}|^2+\varepsilon^2)^\frac{\mu+p}{2}    \right)^\theta
	\end{align*}
	for any $R_2<R_1<8R_2$. Letting $\varepsilon\to 0$ gives
	\begin{align*}
		\fint_{Q_{R_2}^\lambda} |\nabla_y \overline{w}|^{\theta \mu + p+\frac{4}{n}}
		\leq  \frac{C\lambda^{\frac{2(p-2)}{n}}}{(1-R_2/R_1)^{2\theta}}  \left( \lambda^{2-p} \fint_{Q_{R_1}^\lambda} |\nabla_y \overline{w}|^{\mu+2}   +  \fint_{Q_{R_1}^\lambda}  |\nabla_y \overline{w}|^{\mu+p}    \right)^\theta.
	\end{align*}
	 Following the proof of Lemma \ref{z}, one gets
	\begin{equation}\label{Linftywbdry}
	\lambda\|\nabla_y \overline{w}\|_{L^\infty(Q_{R}^\lambda)}\leq C\left(\fint_{Q_{2R}^\lambda}|\lambda\nabla_y \overline{w}|^{\beta_1}\right)^{\frac{\nu_0}{\beta_{1}}}+C\left(\fint_{Q_{2R}^\lambda}|\lambda\nabla_y \overline{w}|^{\beta_1}\right)^{\frac{\nu_1}{\beta_{1}}}.
\end{equation}
	For any $x\in B_{R/2}(0)\cap \Omega$, one has $w(x,t)=\overline{w}(\mathcal{T}x,t)$ and $\nabla w(x,t)=D \mathcal{T}\nabla_y \overline{w}(\mathcal{T}x,t)$. By
	\eqref{subset},  \eqref{DT} and \eqref{Linftywbdry} we obtain
	\begin{equation*}
		\lambda\|\nabla {w}\|_{L^\infty(Q_{R/2}^\lambda)}\leq C\left(\fint_{Q_{4R}^\lambda}|\lambda\nabla {w}|^{\beta_1}\right)^{\frac{\nu_0}{\beta_{1}}}+C\left(\fint_{Q_{4R}^\lambda}|\lambda\nabla {w}|^{\beta_1}\right)^{\frac{\nu_1}{\beta_{1}}}.
	\end{equation*}
Then we complete the proof.
\end{proof}\vspace{0.5cm}\\
We also have the following counterpart of comparison estimates near the boundary.
\begin{lemma}\label{u-wbdry}
	Let $w$ be a weak solution of \eqref{eqwbdry}. For any $\vartheta>\frac{{2n}}{p}-(2+n)$ and any $\eta, \kappa_0\in(0,1)$, let 
		\begin{align*}
		\sigma = p+\vartheta +\frac{p(\vartheta +2)}{n}>\frac{n(2-p)}{p},\ \sigma_1 = p+\frac{n\vartheta }{n+\vartheta +2}>\frac{n(2-p)}{2}.
	\end{align*}
	If $[a]_{4R}\leq \eta^{\frac{n^4}{\kappa_0^4(p-1)^4}}$, $\nabla u\in L^{\sigma_1+\kappa_0}(Q_{4R}^\lambda(x_0,t_0))$, and  $F\in L^{\sigma_1}(Q_{4R}^\lambda(x_0,t_0))$, then\\
(1)	\begin{equation}\label{u-wLpbdry}
	\begin{split}
		& \left(\fint_{{Q}_{4R}^\lambda(x_0,t_0)} |\lambda \nabla(u-w)|^p\right)^{\frac{1}{p}}\leq C(\eta)\left(\fint_{Q_{4R}^\lambda(x_0,t_0)}  |F|^{p}\right)^\frac{1}{p}+\eta\left(\fint_{Q_{4R}^\lambda(x_0,t_0)}|\nabla u|^{\sigma_1}\right)^\frac{1}{\sigma_1}.
	\end{split}
\end{equation}
	(2)\begin{equation}\label{u-wbdrycase1}
		\begin{split}
			& \left[\fint_{t_0-\lambda^{p-2}
				(4R)^2}^{t_0}\left(\fint_{B_{4R}(x_0)\cap \Omega}
			\left(\frac{\lambda}{R}|u-w|\right)^{\sigma}dx\right)^\frac{2}{\sigma}dt\right]
			^\frac{1}{2}\\
			&\quad\leq
			C(\eta,\kappa_0)\left(\fint_{Q_{4R}^\lambda(x_0,t_0)}|\lambda F|^{\sigma_1}\right)^{\frac{n+p}{\sigma_1(n+2)}}+
			\eta\left(\fint_{Q_{4R}^\lambda(x_0,t_0)}|\lambda \nabla u|^{\sigma_1+\kappa_0}\right)^{\frac{n+p}{(\sigma_1+\kappa_0)(n+2)}},
		\end{split}
	\end{equation}
\end{lemma}

In the following, we give a proposition as a boundary version of Proposition \ref{prop1}.
\begin{proposition}\label{prop2}
	Let $\Omega$ be a bounded domain in $\mathbb{R}^n$ with $C^2$ boundary, $u$ be a weak solution of \eqref{e1} in $\Omega_T$, and $w$ be a weak solution to \eqref{eqwbdry}. There exists $R_0>0$ small depending on $\partial \Omega$ such that for any $(x_0,t_0)\in \partial\Omega\times [0,T]$ and any $0<R< R_0$,
	there exists  $\tilde{\mathbf{h}}\in L^\infty({Q}_{R/16}^\lambda(x_0,t_0))$ and  $\frac{n(2-p)}{2}<\sigma_1<n+p$ such that for any $\eta,\kappa_0\in(0,1)$
	\begin{align}
		\lambda\|\tilde{\mathbf{h}}\|_{L^\infty({Q}_{R/16}^\lambda(x_0,t_0))}\leq&
		C G_{\gamma_0}\left(\left(\fint_{{Q}_{4R}^\lambda(x_0,t_0)}|\lambda\nabla u|^{\sigma_1} \right)^{\frac{1}{\sigma_1}}\right)+ C(\eta,\kappa_0)G_{\gamma_0}\left(\left(\fint_{{Q}_{4R}^\lambda(x_0,t_0)}|\lambda F|^{\sigma_1} \right)^{\frac{1}{\sigma_1}}\right)\nonumber\\
		&\quad\quad	+\eta G_{\gamma_0}\left(\left(\fint_{{Q}_{4R}^\lambda(x_0,t_0)}|\lambda\nabla u|^{\sigma_1+\kappa_0} \right)^{\frac{1}{\sigma_1+\kappa_0}}\right),\label{hbdry1}
	\end{align}
	and
	\begin{align}
		\left(\fint_{{Q}_{R/16}^\lambda(x_0,t_0)}|\lambda(\nabla u-\tilde{\mathbf{h}})|^{\sigma_1} \right)^{\frac{1}{\sigma_1}}&\leq \eta G_{\gamma_0}\left(\left(\fint_{{Q}_{4R}^\lambda(x_0,t_0)}|\lambda\nabla u|^{\sigma_1+\kappa_0} \right)^{\frac{1}{\sigma_1+\kappa_0}}\right)\nonumber\\&\qquad+ C(\eta,\kappa_0)G_{\gamma_0}\left(\left(\fint_{{Q}_{4R}^\lambda(x_0,t_0)}|\lambda F|^{\sigma_1} \right)^{\frac{1}{\sigma_1}}\right),\label{hbdry2}
	\end{align}
	with $G_{\gamma_0}$ defined in \eqref{formularG}
	for some $\gamma_0\in (0,1)$, provided $[a]_{4R}\leq \eta^{\frac{n^4}{\kappa_0^4(p-1)^4}} $, $\nabla u\in L^{\sigma_1+\kappa_0}(Q^\lambda_{4R}(x_0,t_0))$ and $F\in L^{\sigma_1}(Q^\lambda_{4R}(x_0,t_0))$.
\end{proposition}
\begin{proof}
	For simplicity, we let $(x_0,t_0)=(0,0)$. Let $\tilde{\mathbf{h}}=\nabla w$ where $w$ is a solution to \eqref{eqwbdry}. It follows from Lemma \ref{lemBdry1} and Lemma \ref{lemBdry2} that
	\begin{align*}
		\lambda\|\tilde{\mathbf{h}}\|_{L^\infty(Q_{R/16}^\lambda)}&\lesssim \left(\fint_{Q_{R/2}^\lambda}|\lambda\nabla w|^{\beta_{1}}\right)^{\frac{\nu_0}{\beta_{1}}}+\left(\fint_{Q_{R/2}^\lambda}|\lambda\nabla w|^{\beta_{1}}\right)^{\frac{\nu_1}{\beta_{1}}}\nonumber\\
		&\lesssim {A}_1^{\frac{\nu_0}{\beta_1}}+{A}_1^{\frac{2\nu_1}{p}}+
		\left[G_{\frac{2p-2}{p}}\left(\fint_{Q^\lambda_{4R}}|\lambda\nabla  w|^{p}\right)\right]^{\frac{\nu_0}{\beta_1}}+	\left[G_{\frac{2p-2}{p}}\left(\fint_{Q^\lambda_{4R}}|\lambda\nabla w|^{p}\right)\right]^{\frac{\nu_1}{\beta_1}},
	\end{align*}
	where $\beta_{1}>\frac{n(2-p)}{2}$ and $\nu_0$, $\nu_1$ are defined in Lemma \ref{z} and $A_1$ is defined by \eqref{def-A1} in Lemma \ref{lemBdry1}.
	By \eqref{u-wbdrycase1} in Lemma \ref{u-wbdry}, for $q_0=\frac{2\beta_1}{p}>\frac{n(2-p)}{p}$, $A_1$ can be estimated by
	\begin{align*}
		A_1
		\leq &C\left[\fint_{-\lambda^{p-2}
			(4R)^2}^0\left(\fint_{B_{4R}\cap\Omega}
		\left(\frac{\lambda}{R}|u-w|\right)^{q_0}dx\right)^\frac{\sigma_2}{q_0}dt\right]
		^\frac{1}{\sigma_2}\\
		&\quad\quad\quad\quad\quad\quad+C\left[\fint_{-\lambda^{p-2}
			(4R)^2}^0\left(\fint_{B_{4R}\cap\Omega}
		\left(\frac{\lambda}{R}|u|\right)^{q_0}dx\right)^\frac{\sigma_2}{q_0}dt\right]
		^\frac{1}{\sigma_2}\\
		\leq &	C(\eta,\kappa_0)\left(\fint_{Q_{4R}^\lambda}|\lambda F|^{\sigma_1}\right)^{\frac{n+p}{\sigma_1(n+2)}}+
		\eta\left(\fint_{Q_{4R}^\lambda}|\lambda \nabla u|^{\sigma_1+\kappa_0}\right)^{\frac{n+p}{(\sigma_1+\kappa_0)(n+2)}}\\
		&\quad\quad\quad\quad\quad\quad+C\left[\fint_{-\lambda^{p-2}
			(4R)^2}^0\left(\fint_{B_{4R}\cap\Omega}
		\left(\frac{\lambda}{R}|u|\right)^{q_0}dx\right)
		^\frac{\sigma_2}{q_0}dt\right]^\frac{1}{\sigma_2},
	\end{align*}
	where
	\[\sigma_1=\frac{q_0(n+p)}{q_0+n+2}> \frac{n(2-p)}{2},\quad\quad \text{and}\quad \sigma_2=\min\{\sigma_1,2\}.\]
	Note that $u=0$ on $B_{4R}\cap\Omega$, then by Poincar\'e--Sobolev's inequality (see \cite{DiBenedetto1}),
	\begin{align*}
		\left[\fint_{-\lambda^{p-2}
			(4R)^2}^0\left(\fint_{B_{4R}\cap\Omega}
		\left(\frac{\lambda}{R}|u|\right)^{q_0}dx\right)
		^\frac{\sigma_2}{q_0}dt\right]^\frac{1}{\sigma_2}&\lesssim \left[\fint_{-\lambda^{p-2}
			(4R)^2}^0\left(\fint_{B_{4R}\cap\Omega}
		|\lambda\nabla u|^{\frac{nq_0}{n+q_0}}dx\right)
		^\frac{\sigma_2(n+q_0)}{nq_0}dt\right]^\frac{1}{\sigma_2}\\
		&\lesssim \left(\fint_{Q^\lambda_{4R}}
		|\lambda\nabla u|^{\sigma_1}\right)^\frac{1}{\sigma_1}.
	\end{align*}
	Then we can follow the proof of Proposition \ref{prop1} and use Lemma \ref{lemBdry1}, Lemma \ref{lemBdry2} and Lemma \ref{u-wbdry} to obtain \eqref{hbdry1}. Here we omit the details.
	
	Next, using interpolation inequality and Young's inequality, one has
	\begin{align*}
		\left(\fint_{Q_{R/16}^\lambda}|\lambda(\nabla u-\tilde{\mathbf{h}})|^{\sigma_1} \right)^{\frac{1}{\sigma_1}}
		&\leq C(\eta,\kappa_0)\left(\fint_{Q_{R/16}^\lambda}|\lambda\nabla(u-w)|^{p} \right)^{\frac{1}{p}}+\eta^2\left(\fint_{Q_{R/16}^\lambda}|\lambda\nabla u|^{\sigma_1+\kappa_0} \right)^{\frac{1}{\sigma_1+\kappa_0}}\\&\quad\quad\quad\quad+\eta^2\left(\fint_{Q_{R/16}^\lambda}|\lambda\nabla w|^{\sigma_1+\kappa_0} \right)^{\frac{1}{\sigma_1+\kappa_0}}.
	\end{align*}
By \eqref{u-wLpbdry}, \eqref{hbdry1} and H\"{o}lder's inequality, we have
\begin{align*}
	&\left(\fint_{Q_{R/16}^\lambda}|\lambda(\nabla u-\tilde{\mathbf{h}})|^{\sigma_1} \right)^{\frac{1}{\sigma_1}}\\
	&\leq C(\eta,\kappa_0) \left(\fint_{Q_{R/2}^\lambda}  |\lambda F|^{p}\right)^\frac{1}{p} +2\eta^2\left(\fint_{Q_{R/2}^\lambda}|\lambda\nabla u|^{\sigma_1+\kappa_0} \right)^{\frac{1}{\sigma_1+\kappa_0}}+\eta^2 \lambda\|\mathbf{h}\|_{L^\infty(Q_{R/16}^\lambda)}   \\
	&\leq C(\eta,\kappa_0)G_{\gamma_0}\left(\left(\fint_{Q_{4R}^\lambda}|\lambda F|^{\sigma_1} \right)^{\frac{1}{\sigma_1}}\right)+\eta G_{\gamma_0}\left(\left(\fint_{Q_{4R}^\lambda}|\lambda\nabla u|^{\sigma_1+\kappa_0} \right)^{\frac{1}{\sigma_1+\kappa_0}}\right).
\end{align*}
Then we finish the proof of \eqref{hbdry2}.
\end{proof}
\begin{proposition}\label{prop22}
	Let $\Omega$ be a bounded domain in $\mathbb{R}^n$ with $C^2$ boundary, $u$ be a weak solution of \eqref{e1} in $\Omega_T$, and $w$ be a weak solution to \eqref{eqwbdry}. There exists $R_0>0$ small depending on $\partial \Omega$ such that for any $(x_0,t_0)\in \partial\Omega\times[0,T]$ and any $0<R< R_0$,
	there exists  $\tilde{\mathbf{h}}\in L^\infty({Q}_{R/2}^\lambda(x_0,t_0))$  and $\frac{n(2-p)}{2}<\sigma_1<n+p$ such that for any $\eta,\kappa_0\in(0,1)$
	\begin{align*}
		\lambda\|\tilde{\mathbf{h}}\|_{L^\infty({Q}_{R/2}^\lambda(x_0,t_0))}\leq&
		C G_{\gamma_0}\left(\left(\fint_{{Q}_{2R}^\lambda(x_0,t_0)}|\lambda\nabla u|^{\sigma_1} \right)^{\frac{1}{\sigma_1}}\right)+ C(\eta,\kappa_0)G_{\gamma_0}\left(\left(\fint_{{Q}_{2R}^\lambda(x_0,t_0)}|\lambda F|^{\sigma_1} \right)^{\frac{1}{\sigma_1}}\right)\\
		&\quad\quad	+\eta G_{\gamma_0}\left( \left(\fint_{{Q}_{2R}^\lambda(x_0,t_0)}|\lambda\nabla u|^{\sigma_1+\kappa_0} \right)^{\frac{1}{\sigma_1+\kappa_0}}\right),
	\end{align*}
	and
	\begin{align*}
		\left(\fint_{{Q}_{R/2}^\lambda(x_0,t_0)}|\lambda(\nabla u-\tilde{\mathbf{h}})|^{\sigma_1} \right)^{\frac{1}{\sigma_1}}&\leq \eta G_{\gamma_0}\left(\left(\fint_{{Q}_{2R}^\lambda(x_0,t_0)}|\lambda\nabla u|^{\sigma_1+\kappa_0} \right)^{\frac{1}{\sigma_1+\kappa_0}}\right)\\&\qquad+ C(\eta,\kappa_0)G_{\gamma_0}\left(\left(\fint_{{Q}_{2R}^\lambda(x_0,t_0)}|\lambda F|^{\sigma_1} \right)^{\frac{1}{\sigma_1}}\right),
	\end{align*}
with $G_{\gamma_0}$ defined in \eqref{formularG}
for some $\gamma_0\in (0,1)$, provided $[a]_{2R}\leq \eta^{\frac{n^4}{\kappa_0^4(p-1)^4}} $, $\nabla u\in L^{\sigma_1+\kappa_0}({Q}_{2R}^\lambda(x_0,t_0))$ and $F\in L^{\sigma_1}({Q}_{2R}^\lambda(x_0,t_0))$.
\end{proposition}
\begin{proof} For any $\lambda, r>0$, define the cube $\mathcal{Q}^\lambda_r(x,t)=\{x+[-r,r]^n\}\times[t-\lambda^{p-2}r^2,t]$.
	There exists $m>0$ indenpendent of $R$ such that for any $(x_0,t_0)$, there exist $\{(x_k,t_k)\}_{k=1}^m\subset Q^\lambda_{R/2}(x_0,t_0)$ such that
	$$
Q^\lambda_{R/2}(x_0,t_0)\subset \cup_{k=1}^m\mathcal{Q}^{\lambda,k}_{R/100},\quad\quad\text{and}\quad \mathring{\mathcal{Q}}^{\lambda,k_1}_{R/100}\cap \mathring{\mathcal{Q}}^{\lambda,k_2}_{R/100}=\emptyset,\ k_1\neq k_2,
	$$
	where we denote $\mathcal{Q}^{\lambda,k}_{r}=\mathcal{Q}^{\lambda}_{r}(x_k,t_k)$, and $\mathring{\mathcal{Q}}^{\lambda,k}_{r}$ is the interior of $\mathcal{Q}^{\lambda,k}_{r}$, for $k=1,\cdots,m$ and any $r>0$.\\
	By Proposition \ref{prop2}, for any $k=1,\cdots,m$, there exists $\tilde{\mathbf{h}}_k\in L^\infty(Q^{\lambda,k}_{R/50})$ such that
	\begin{align*}
		\lambda\|\tilde{\mathbf{h}}_k\|_{L^\infty(\mathcal{Q}_{R/100}^{\lambda,k})}\leq	\lambda\|\tilde{\mathbf{h}}_k\|_{L^\infty({Q}_{R/50}^{\lambda,k})}&\leq
		C G_{\gamma_0}\left(\left(\fint_{{Q}_{R}^{\lambda,k}}|\lambda\nabla u|^{\sigma_1} \right)^{\frac{1}{\sigma_1}}\right)+ C(\eta,\kappa_0)G_{\gamma_0}\left(\left(\fint_{{Q}_{R}^{\lambda,k}}|\lambda F|^{\sigma_1} \right)^{\frac{1}{\sigma_1}}\right)\\
		&\quad\quad\quad+\eta G_{\gamma_0}\left(\left(\fint_{{Q}_{R}^{\lambda,k}}|\lambda\nabla u|^{\sigma_1+\kappa_0} \right)^{\frac{1}{\sigma_1+\kappa_0}}\right)
	\end{align*}
	and
	\begin{align*}
		\left(\fint_{\mathcal{Q}_{R/100}^{\lambda,k}}|\lambda(\nabla u-\tilde{\mathbf{h}}_k)|^{\sigma_1} \right)^{\frac{1}{\sigma_1}}&\leq 	\left(\fint_{Q_{R/50}^{\lambda,k}}|\lambda(\nabla u-\tilde{\mathbf{h}}_k)|^{\sigma_1} \right)^{\frac{1}{\sigma_1}}\\
		&\leq \eta^2 G_{\gamma_0}\left(\left(\fint_{{Q}_{R}^{\lambda,k}}|\lambda\nabla u|^{\sigma_1+\kappa_0} \right)^{\frac{1}{\sigma_1+\kappa_0}}\right) + C(\eta,\kappa_0)G_{\gamma_0}\left(\left(\fint_{{Q}_{R}^{\lambda,k}}|\lambda F|^{\sigma_1} \right)^{\frac{1}{\sigma_1}}\right)
	\end{align*}
	for some $\gamma_0\in (0,1)$ and $\kappa_0>0$.
	Denote $\mathbf{h}=\sum_{k=1}^m \mathbf{1}_{\mathcal{Q}^{\lambda,k}_{R/100}}\tilde{\mathbf{h}}_k$. Then
	\begin{align*}
		\lambda\|{\mathbf{h}}\|_{L^\infty({Q}_{R/2}^{\lambda})}&\leq \sum_{k=1}^m\lambda\|\tilde{\mathbf{h}}_k\|_{L^\infty(\mathcal{Q}_{R/100}^{\lambda,k})}\\
		&\leq	C G_{\gamma_0}\left(\left(\fint_{{Q}_{2R}^{\lambda}}|\lambda\nabla u|^{\sigma_1} \right)^{\frac{1}{\sigma_1}}\right)+ C(\eta,\kappa_0)G_{\gamma_0}\left(\left(\fint_{{Q}_{2R}^{\lambda}}|\lambda F|^{\sigma_1} \right)^{\frac{1}{\sigma_1}}\right)\\
		&\quad\quad\quad+\eta G_{\gamma_0}\left(\left(\fint_{{Q}_{2R}^{\lambda}}|\lambda\nabla u|^{\sigma_1+\kappa_0} \right)^{\frac{1}{\sigma_1+\kappa_0}}\right),
	\end{align*}
	and
	\begin{align*}
		&\left(\fint_{{Q}_{R/2}^{\lambda}}|\lambda(\nabla u-{\mathbf{h}})|^{\sigma_1} \right)^{\frac{1}{\sigma_1}}\leq C\sum_{k=1}^m	\left(\fint_{\mathcal{Q}_{R/100}^{\lambda,k}}|\lambda(\nabla u-{\tilde{\mathbf{h}}_k})|^{\sigma_1} \right)^{\frac{1}{\sigma_1}}
		\\&\quad 	\leq C\eta^2 \sum_{k=1}^mG_{\gamma_0}\left(\left(\fint_{{Q}_{R}^{\lambda,k}}|\lambda\nabla u|^{\sigma_1+\kappa_0} \right)^{\frac{1}{\sigma_1+\kappa_0}}\right)+ C(\eta,\kappa_0)\sum_{k=1}^mG_{\gamma_0}\left(\left(\fint_{{Q}_{R}^{\lambda,k}}|\lambda F|^{\sigma_1} \right)^{\frac{1}{\sigma_1}}\right) \\
		&\quad 	\leq\eta G_{\gamma_0}\left(\left(\fint_{{Q}_{2R}^{\lambda}}|\lambda\nabla u|^{\sigma_1+\kappa_0} \right)^{\frac{1}{\sigma_1+\kappa_0}}\right)+ C(\eta,\kappa_0)G_{\gamma_0}\left(\left(\fint_{{Q}_{2R}^{\lambda}}|\lambda F|^{\sigma_1} \right)^{\frac{1}{\sigma_1}}\right),
	\end{align*}
	which completes the proof.
\end{proof}\vspace{0.5cm}\\
Combining Proposition \ref{prop1} and Proposition \ref{prop22}, we obtain Theorem \ref{thm0}.
\section{Proof of Theorem \ref{thm1}}
In this section, we will prove Theorem \ref{thm1}.
Let us define for $\lambda>0$, $$\widetilde{Q}_{\rho}^\lambda(x_0,t_0)=B_\rho(x_0)\times(-\lambda^{p-2}\rho^2/2+t_0,\lambda^{p-2}\rho^2/2 +t_0) .$$
It is easy to check that $\widetilde{Q}_{\rho}^\lambda(x_0,t_0)=Q^\lambda_{\rho}(x_0,t_0+\lambda^{p-2}\rho^2/2)$. Hence all the results in the previous section also hold for $\widetilde{Q}_{\rho}^\lambda(x_0,t_0)$.\\
 Let $\mathbf{M}^\lambda$ denote the Hardy--Littlewood $\lambda$-maximal function defined for each locally integrable function $f$ in $\mathbb{R}^{n}\times\mathbb{R}$ by
\begin{align}\label{Mlamb}
	\mathbf{M}^\lambda(f)(x,t)=\sup_{\rho>0}\fint_{\widetilde{Q}^\lambda_{\rho}(x,t)}|f(y,s)|dyds,\ \forall (x,t)\in \mathbb{R}^{n+1}.
\end{align}
We first recall an important technical lemma. It is  the Krylov-Safonov Lemma (see \cite{Byun05,Byun07,MePhuc11}). 
\begin{lemma}\cite[Lemma 1.50]{QH1}\label{lem-eo}
	Let $\Omega$ be a bounded $C^1$ domain. Suppose that the sequence of balls $\{B_{r}(y_i)\}_{i=1}^L$ with centers $y_i\in \overline{\Omega}$ and a common radius $r>0$ covers $\Omega$. Set $s_i=T-i\lambda^{p-2}r^2/2$ for all $i=0,1,\cdots,[\frac{2T}{\lambda^{p-2}r^2}]$. Let $E\subset O\subset \Omega_T$ be
	 measurable sets for which there exists $0<\varepsilon<1$ such that $|E|<\varepsilon |\widetilde{Q}^\lambda_{r}(y_i,s_j)|$ for all $i=1,\cdots,L$ and $j=0,1,\cdots,[\frac{2T}{\lambda^{p-2}r^2}]$; and for all $(x,t)\in \Omega_T$, $\rho\in (0,2r]$, we have $\widetilde{Q}^\lambda_{\rho}(x,t)\cap \Omega_T\subset O$ if $|E\cap \widetilde{Q}^\lambda_\rho(x,t)|\geq \varepsilon |\widetilde{Q}^\lambda_\rho(x,t)|$. Then $|E|\leq C\varepsilon |O|$ for a constant $C$ depending only on $n$.
\end{lemma}
\begin{lemma}\label{cover}
Let $\sigma_1, \sigma_1+\kappa_0$ be those in Proposition \ref{prop1}.	Let  $u\in L^{\sigma_1+\kappa_0}(0,T,W^{1,\sigma_1+\kappa_0}_0(\Omega))$ be a weak solution of \eqref{e1}.  Then for any $\varepsilon>0$, one can find constants $\delta=\delta(n,p,\varepsilon,\sigma_1,\kappa_0)\in (0,1)$, $\delta_1=\delta_1(n,p,\varepsilon,\sigma_1,\kappa_0)\in (0,1)$ and $\Lambda=\Lambda(n,p,\sigma_1)>0$ such that if $[a]_{R_0}\leq \delta_1$ for some $R_0>0$, then for any $0<\lambda\leq \lambda_0$ with
	\begin{equation*}
		\lambda_0=(\varepsilon\Lambda^{\sigma_1}/C_0)^{\frac{1}{\sigma_1+2-p}}\|\nabla u\|_{L^{\sigma_1}(\Omega_T)}^{-\frac{\sigma_1}{\sigma_1+2-p}},
	\end{equation*}
there holds
	\begin{align*}
		&\big|\{ (\lambda{\bf M}^\lambda (|\nabla u|^{\sigma_1} ))^{\frac{1}{\sigma_1}}>\Lambda,\lambda({\bf M}^\lambda (|F|^{\sigma_1} ))^{\frac{1}{\sigma_1}}\leq \delta\}\cap U_{\lambda,
			\varepsilon}\cap\Omega_T\big|\\
		&\qquad\leq C\varepsilon \big|\{\lambda({\bf M}^\lambda (|\nabla u|^{\sigma_1} ))^{\frac{1}{\sigma_1}}>1\}\cap\Omega_T\big|.
	\end{align*}
  Here \[U_{\lambda,\varepsilon}=\{\lambda({\bf M}^\lambda (|\nabla u|^{\sigma_1+\kappa_0} ))^{\frac{1}{\sigma_1+\kappa_0}}\leq \varepsilon^{-1} \},\]
and the constants $C$, $C_0$ depend on $n,p$.
\end{lemma}
\begin{proof}
	We follow some ideas in \cite{Hung2,QH1,HungPhuc4,HungPhuc3}.	Let
	\begin{equation*}
		O_{\lambda}=\{\lambda({\bf M}^\lambda (|\nabla u|^{\sigma_1} ))^{\frac{1}{\sigma_1}}>1\}\cap\Omega_T,
	\end{equation*}
and
 \begin{align*}
		E_{\lambda,\varepsilon}=\{ \lambda({\bf M}^\lambda (|\nabla u|^{\sigma_1} ))^{\frac{1}{\sigma_1}}>\Lambda,\lambda({\bf M}^\lambda (|F|^{\sigma_1} ))^{\frac{1}{\sigma_1}}\leq \delta\}\cap U_{\lambda,
		\varepsilon}\cap\Omega_T.
	\end{align*}
Let $T_0=\text{diam} (\Omega) + T^{1/2}$. Let $\{y_i\}_{i=1}^L\subset \Omega$ and a ball $B_0$ with radius $2T_0$ such that
\[\Omega \subset \bigcup_{i=1}^L B_{r_0}(y_i)\subset B_0,\]
where $r_0= \min\{T_0,R_0/100\}$ with $R_0$ defined by Proposition \ref{prop22}. Let $s_j=T-j\lambda^{p-2}r_0^2/2$ for $j=0,1,\cdots, [\frac{2T}{\lambda^{p-2}r_0^2}]+1$ and $Q_{2T_0}^\lambda=B_0\times (T-\lambda^{p-2}(2T_0)^2, T)$. So,
\[\Omega_T\subset \bigcup_{i,j}Q_{r_0}^\lambda(y_i,s_j)\subset Q_{2T_0}^\lambda.\]

We verify that
\begin{equation}\label{cond1}
	|E_{\lambda,\varepsilon}|\leq \frac{\varepsilon}{2} |\widetilde{Q}_{r_0}^\lambda(y_i,s_j)|,\  \forall i=1,\cdots,L,\quad\text{and }\quad j=0,1,\cdots,\left[\frac{2T}{\lambda^{p-2}r^2}\right]+1
\end{equation}
holds for $
	\lambda\leq  (\varepsilon\Lambda^{\sigma_1}/C_0)^{\frac{1}{\sigma_1+2-p}}\|\nabla u\|_{L^{\sigma_1}(\Omega_T)}^{-\frac{\sigma_1}{\sigma_1+2-p}}
$ with $C_0=C_0(n,p,r_0)$ being a large constant.
In fact, for such $\lambda$, one has
\begin{align*}
	&|E_{\lambda,\varepsilon}|\leq\frac{C\lambda^{\sigma_1}}{\Lambda^{\sigma_1}}\int_{\Omega_T}|\nabla u|^{\sigma_1}\leq \frac{C}{C_0}\varepsilon r_0^{n+2}\lambda^{p-2}\leq \varepsilon |\widetilde{Q}_{r_0}^\lambda(y_i,s_j)|.
\end{align*}
Next, we verify that for all $(x,t)\in \Omega_T$ and $r\in (0,2r_0]$ and $\lambda>0$, we have
\begin{align}\label{cond2}
	|E_{\lambda,\varepsilon}\cap \widetilde{Q}_r^\lambda(x,t)|\geq \varepsilon |\widetilde{Q}_r^\lambda(x,t)|\Rightarrow \widetilde{Q}_r^\lambda(x,t)\cap\Omega_T\subset O_\lambda.
\end{align}
 Indeed, take $(x,t)\in \Omega_T$ and $0<r\leq 2r_0$. By contraposition, assume that $\widetilde{Q}_r^\lambda(x,t)\cap\Omega_T\cap O_\lambda^c\neq \emptyset$ and $E_{\lambda,\varepsilon}\cap \widetilde{Q}_r^\lambda(x,t)\neq \emptyset$, i.e., there exist $(x_1,t_1),(x_2,t_2)\in \widetilde{Q}_r^\lambda(x,t)\cap\Omega_T$ such that
\begin{equation}\label{xt1}
	\lambda({\bf M}^\lambda (|\nabla u|^{\sigma_1} ))^{\frac{1}{\sigma_1}}(x_1,t_1)\leq 1,
\end{equation}
and
\begin{equation}\label{xt2}
	\lambda({\bf M}^\lambda (|F|^{\sigma_1} ))^{\frac{1}{\sigma_1}}(x_2,t_2)\leq \delta,\ \lambda({\bf M}^\lambda (|\nabla u|^{\sigma_1+\kappa_0} ))^{\frac{1}{\sigma_1+\kappa_0}}(x_2,t_2)\leq \varepsilon^{-1}.
\end{equation}
We need to prove that
\begin{align}\label{Elambeps}
	|E_{\lambda,\varepsilon}\cap \widetilde{Q}_r^\lambda(x,t)|< \varepsilon |\widetilde{Q}_r^\lambda(x,t)|.
\end{align}
It follows from \eqref{xt1} that
\begin{equation}\nonumber
	\lambda({\bf M}^\lambda (|\nabla u|^{\sigma_1} ))^{\frac{1}{\sigma_1}}(y,s)\leq \max\{\lambda(\textbf{M}^\lambda(\chi_{\widetilde{Q}_{2r}^\lambda(x,t)}|\nabla u|^{\sigma_1}))^{\frac{1}{\sigma_1}}(y,s),3^{\frac{n+2}{\sigma_1}}\},\ \forall (y,s)\in \widetilde{Q}_r^\lambda(x,t).
\end{equation}
Therefore, for all $\lambda>0$ and $\Lambda\geq 3^\frac{n+2}{\sigma_1}$,
\begin{align}\label{Elamb}
	E_{\lambda,\varepsilon}\cap \widetilde{Q}_r^\lambda(x,t)=\{\lambda(\textbf{M}^\lambda(\chi_{\widetilde{Q}_{2r}^\lambda(x,t)}|\nabla u|^{\sigma_1}))^{\frac{1}{\sigma_1}}\geq \Lambda, \lambda({\bf M}^\lambda (|F|^{\sigma_1} ))^{\frac{1}{\sigma_1}}\leq \delta\}\cap U_{\lambda,\varepsilon}\cap\Omega_T\cap\widetilde{Q}_r^\lambda(x,t).
\end{align}
To prove \eqref{Elambeps}, we separately consider the case $B_{8r}(x)\subset\subset \Omega$ and the case $\overline{B_{8r}(x)}\cap \Omega^c\neq \emptyset$.

We first consider the case $B_{8r}\subset\subset \Omega$. Using Proposition \ref{prop1} in $Q_{8r}^\lambda(x,t_0)$ where $t_0=\min\{t+\lambda^{p-2}2r^2,T\}$, there exists $\mathbf{h}$ such that
\begin{align*}
	\lambda\|\mathbf{h}\|_{L^\infty({Q}_{2r}^\lambda(x,t_0))}\leq&
	C G_{\gamma_0}\left(\left(\fint_{{Q}_{8r}^\lambda(x,t_0)}|\lambda\nabla u|^{\sigma_1} \right)^{\frac{1}{\sigma_1}}\right)+ C(\eta,\kappa_0)G_{\gamma_0}\left(\left(\fint_{{Q}_{8r}^\lambda(x,t_0)}|\lambda F|^{\sigma_1} \right)^{\frac{1}{\sigma_1}}\right)\\
	&\quad\quad\quad +\eta G_{\gamma_0}\left(\left(\fint_{{Q}_{8r}^\lambda(x,t_0)}|\lambda\nabla u|^{\sigma_1+\kappa_0} \right)^{\frac{1}{\sigma_1+\kappa_0}}\right),
\end{align*}
and for any $\eta>0$,
\begin{align*}
	\left(\fint_{{Q}_{2r}^\lambda(x,t_0)}|\lambda(\nabla u-\mathbf{h})|^{\sigma_1} \right)^{\frac{1}{\sigma_1}}&\leq \eta
	G_{\gamma_0}\left(\left(\fint_{{Q}_{8r}^\lambda(x,t_0)}|\lambda\nabla u|^{\sigma_1+\kappa_0} \right)^{\frac{1}{\sigma_1+\kappa_0}}\right)\\&\quad \quad\quad+ C(\eta,\kappa_0)G_{\gamma_0}\left(\left(\fint_{{Q}_{8r}^\lambda(x,t_0)}|\lambda F|^{\sigma_1} \right)^{\frac{1}{\sigma_1}}\right)
\end{align*}
provided $[a]_{R_0}\leq \eta^{\frac{n^4}{\kappa_0^4(p-1)^4}}$.
Thanks to \eqref{xt1} and \eqref{xt2} with $(x_1,t_1),(x_2,t_2)\in \widetilde{Q}_r^\lambda(x,t)\cap\Omega_T$, we get that
\begin{equation}\label{DvLinf}
	\begin{split}
	\lambda\|\mathbf{h}\|_{L^\infty(\widetilde{Q}_{2r}^\lambda(x,t))}&\leq
	C G_{\gamma_0}\left(\left(\fint_{\widetilde{Q}_{17r}^\lambda(x_1,t_1)}|\lambda\nabla u|^{\sigma_1} \right)^{\frac{1}{\sigma_1}}\right)+ C(\eta,\kappa_0)G_{\gamma_0}\left(\left(\fint_{\widetilde{Q}_{17r}^\lambda(x_2,t_2)}|\lambda F|^{\sigma_1} \right)^{\frac{1}{\sigma_1}}\right)\\
	&\quad\qquad\quad\quad\quad\quad+\eta G_{\gamma_0}\left(\left(\fint_{\widetilde{Q}_{17r}^\lambda(x_2,t_2)}|\lambda\nabla u|^{\sigma_1+\kappa_0} \right)^{\frac{1}{\sigma_1+\kappa_0}}\right)\\
	&\leq C G_{\gamma_0}\left( \lambda\left(\mathbf{M}^\lambda (|\nabla u|^{\sigma_1})(x_1,t_1)\right)^\frac{1}{\sigma_1} \right)+ C(\eta,\kappa_0) G_{\gamma_0}\left(\lambda \left(\mathbf{M}^\lambda (|F|^{\sigma_1})(x_2,t_2)\right)^\frac{1}{\sigma_1} \right)\\
	&\quad\quad\quad+C \eta G_{\gamma_0} \left( \lambda \left({\bf M}^\lambda (|\nabla u|^{\sigma_1+\kappa_0} )(x_2,t_2)\right)^{\frac{1}{\sigma_1+\kappa_0}}\right)\\
	&\leq C(1+C(\eta,\kappa_0)\delta^{\gamma_0}+\eta\varepsilon^{-1/\gamma_0})\leq C_1,
	\end{split}
\end{equation}
provided $C(\eta,\kappa_0)\delta^{\gamma_0},\eta\varepsilon^{-1/{\gamma_0}}\leq 1$, 
and
\begin{align}
	&\left(\fint_{\widetilde{Q}_{2r}^\lambda(x,t)}|\lambda(\nabla u-\mathbf{h})|^{\sigma_1} \right)^{\frac{1}{\sigma_1}}\nonumber\\&\leq C_2\eta
	G_{\gamma_0}\left(\left(\fint_{\widetilde{Q}_{17r}^\lambda(x_2,t_2)}|\lambda\nabla u|^{\sigma_1+\kappa_0} \right)^{\frac{1}{\sigma_1+\kappa_0}}\right)\nonumber+ C_2C(\eta,\kappa_0)G_{\gamma_0}\left(\left(\fint_{\widetilde{Q}_{17r}^\lambda(x_2,t_2)}|\lambda F|^{\sigma_1} \right)^{\frac{1}{\sigma_1}}\right)\nonumber\\
	&\leq C_2	\eta G_{\gamma_0} \left( \lambda \left({\bf M}^\lambda (|\nabla u|^{\sigma_1+\kappa_0} )(x_2,t_2)\right)^{\frac{1}{\sigma_1+\kappa_0}}\right)
	+C_2 C(\eta,\kappa_0)	G_{\gamma_0}\left( \lambda\left(\mathbf{M}^\lambda (|F|^{\sigma_1})(x_2,t_2) \right)^\frac{1}{\sigma_1}\right)\nonumber\\
	& \leq C_2\eta\varepsilon^{-1/{\gamma_0}}+C_2C(\eta,\kappa_0)\delta^{\gamma_0}.\label{Du-Dv}
\end{align}
Note that by \eqref{Elamb} we find
\begin{equation}\label{E-lambda}
	\begin{split}
		|E_{\lambda,\varepsilon}\cap \widetilde{Q}_r^\lambda(x,t)|\leq&\left| \left\{\lambda\left(\textbf{M}^\lambda(\chi_{\widetilde{Q}_{2r}^\lambda(x,t)}|\nabla u-\mathbf{h}|^{\sigma_1})\right)^{\frac{1}{\sigma_1}}\geq \Lambda/4\right\}\cap\widetilde{Q}_r^\lambda(x,t) \right|\\
		&\quad+\left| \left\{\lambda\left(\textbf{M}^\lambda(\chi_{\widetilde{Q}_{2r}^\lambda(x,t)}|\mathbf{h}|^{\sigma_1})\right)^{\frac{1}{\sigma_1}}\geq \Lambda/4\right\}\cap\widetilde{Q}_r^\lambda(x,t) \right|.
	\end{split}
\end{equation}
On the other hand, in view of \eqref{DvLinf} we see that for $\Lambda\geq  8\max\{3^\frac{n+2}{\sigma_1},C_1\}$ ($C_1$ is the constant in \eqref{DvLinf}), it holds that
\[\left| \left\{\lambda\left(\textbf{M}^\lambda(\chi_{\widetilde{Q}_{2r}^\lambda(x,t)}|\mathbf{h}|^{\sigma_1})\right)^{\frac{1}{\sigma_1}}\geq \Lambda/4\right\}\cap\widetilde{Q}_r^\lambda(x,t) \right|=0.\]
Thus we deduce from \eqref{Du-Dv} and \eqref{E-lambda} that
\begin{align*}
	&|E_{\lambda,\varepsilon}\cap \widetilde{Q}_r^\lambda(x,t)|\leq C_3\lambda^{\sigma_1}\int_{\widetilde{Q}_{2r}^\lambda(x,t)}|\nabla u-\mathbf{h}|^{\sigma_1}
	\leq C_3\left(C_2\eta\varepsilon^{-1/{\gamma_0}}+C_2C(\eta,\kappa_0)\delta^{\gamma_0}\right)^{\sigma_1}|\widetilde{Q}_{r}^\lambda(x,t)|.
\end{align*}
For any fixed $\varepsilon>0$, we first take $\eta$ so small that $C_2\eta\varepsilon^{-1/{\gamma_0}}< \frac{\varepsilon}{4C_3}$. Then fix this $\eta$, we take $\delta$ small to get $C_2C(\eta,\kappa_0)\delta^{\gamma_0}<\frac{\varepsilon}{4C_3}$, and take 
$\delta_1=\eta^{\frac{n^4}{\kappa_0^4(p-1)^4}}$.
 After that, we obtain
\begin{align*}
	&|E_{\lambda,\varepsilon}\cap \widetilde{Q}_r^\lambda(x,t)|<\varepsilon |\widetilde{Q}_{r}^\lambda(x,t)|.
\end{align*}

Next we consider the case $\overline{B_{8r}(x)}\cap \Omega^c\neq \emptyset$. Let $x_3\in \partial\Omega$ be such that $|x_3-x|=\operatorname{dist}(x,\partial\Omega)$. Set $t_0=\min\{t+\lambda^{p-2}2r^2,T\}$. We have
\begin{align*}
	{Q}_{2r}^\lambda(x,t_0)\subset {Q}_{10r}^\lambda(x_3,t_0)\subset {Q}_{40r}^\lambda(x_3,t_0)\subset \widetilde{Q}_{80r}^\lambda(x_3,t)\subset \widetilde{Q}_{88r}^\lambda(x,t)\subset \widetilde{Q}_{89r}^\lambda(x_1,t_1)
\end{align*}
and
\begin{align*}
	{Q}_{40r}^\lambda(x_3,t_0)\subset \widetilde{Q}_{80r}^\lambda(x_3,t)\subset \widetilde{Q}_{88r}^\lambda(x,t)\subset \widetilde{Q}_{89r}^\lambda(x_2,t_2).
\end{align*}
Applying Proposition \ref{prop22} in  $ {Q}_{40r}^\lambda(x_3,t_0)$, there exists $\mathbf{h}$ such that
\begin{align*}
	\lambda\|\mathbf{h}\|_{L^\infty(\widetilde{Q}_{2r}^\lambda(x,t))}&\leq\lambda\|\mathbf{h}\|_{L^\infty({Q}_{10r}^\lambda(x_3,t_0))}\\&\leq
	C G_{\gamma_0}\left(\left(\fint_{\widetilde{Q}_{89r}^\lambda(x_1,t_1)}|\lambda\nabla u|^{\sigma_1} \right)^{\frac{1}{\sigma_1}}\right)+ C(\eta,\kappa_0)G_{\gamma_0}\left(\left(\fint_{\widetilde{Q}_{89r}^\lambda(x_2,t_2)}|\lambda F|^{\sigma_1} \right)^{\frac{1}{\sigma_1}}\right)\\
	&\quad\quad\quad+	\eta G_{\gamma_0}\left(\left(\fint_{\widetilde{Q}_{89r}^\lambda(x_2,t_2)}|\lambda\nabla u|^{\sigma_1+\kappa_0} \right)^{\frac{1}{\sigma_1+\kappa_0}}\right),
\end{align*}
and for any $\eta>0$,
\begin{align*}
	&\left(\fint_{\widetilde{Q}_{2r}^\lambda(x,t)}|\lambda(\nabla u-\mathbf{h})|^{\sigma_1} \right)^{\frac{1}{\sigma_1}}\leq \left(\fint_{{Q}_{10r}^\lambda(x_3,t_0)}|\lambda(\nabla u-\mathbf{h})|^{\sigma_1} \right)^{\frac{1}{\sigma_1}}\\&\quad\quad \leq\eta
	G_{\gamma_0}\left(\left(\fint_{\widetilde{Q}_{89r}^\lambda(x_2,t_2)}|\lambda\nabla u|^{\sigma_1+\kappa_0} \right)^{\frac{1}{\sigma_1+\kappa_0}}\right)+ C(\eta,\kappa_0)G_{\gamma_0}\left(\left(\fint_{\widetilde{Q}_{89r}^\lambda(x_2,t_2)}|\lambda F|^{\sigma_1} \right)^{\frac{1}{\sigma_1}}\right).
\end{align*}
As above, we can also obtain
\begin{align*}
	|E_{\lambda,\varepsilon}\cap \widetilde{Q}_r^\lambda(x,t)|\leq C\left(C_2\eta\varepsilon^{-1/{\gamma_0}}+C_2C(\eta,\kappa_0)\delta^{\gamma_0}\right)^{\sigma_1}|\widetilde{Q}_{r}^\lambda(x,t)|<\varepsilon |\widetilde{Q}_{r}^\lambda(x,t)|.
\end{align*}

Using \eqref{cond1} and \eqref{cond2}, we can now apply Lemma \ref{lem-eo} with $E=E_{\lambda,\varepsilon}$ and $O=O_\lambda$ to complete the proof.
\end{proof}
\begin{remark}
	Note that when $p>\frac{2n}{n+2}$, by Remark \ref{reth0} we obtain Lemma \ref{cover} with $\sigma_1=p$ and $\kappa_0=0$.
\end{remark}

We recall an useful lemma which concerns weak (1,1)-estimates for $\lambda$-maximal function. The proof can be found in \cite{Byun21}.
\begin{lemma} \cite[Lemma 2.12]{Byun21}
	If $f\in L^1(\mathbb{R}^{n+1})$, then there exists a constant $c=c(n)\geq 1$ such that
\begin{align}\label{M}
	\left|\{(y,s)\in \mathbb{R}^{n+1}:\mathbf{M}^\lambda f(y,s)>2\alpha\}\right|\leq\frac{c}{\alpha}\int_{\{ |f|>\alpha\}}|f(x,t)|dxdt,
\end{align}
for any $\alpha>0$.
\end{lemma}
\begin{proof}[Proof of Theorem \ref{thm1}] 
	For any $q>\frac{n(2-p)}{2}$, we can find $\sigma_1, \kappa_0>0$ such that $q>\sigma_1+\kappa_0$  and $\frac{n(2-p)}{2}<\sigma_1<n+p$. Next, we divide the proof into two steps.  \\ \emph{Step 1}. Assume $\nabla u\in L^{\sigma_1+\kappa_0}(\Omega_T)$. For any $\varepsilon>0$, let $\delta, \delta_1, \Lambda, \lambda_0$ be as defined in Lemma \ref{cover}.
	One has
	\begin{equation}\label{I1I2}
		\begin{split}
			&\int_{\Omega_T}|\nabla u|^{q}=q\Lambda^{q}\int_{0}^{\infty}\alpha^{q-1}\left|\{|\nabla u|>\Lambda\alpha\}\cap\Omega_T\right|d\alpha\\
			&=q\Lambda^{q}\int_{0}^{\alpha_0}\alpha^{q-1}\left|\{|\nabla u|>\Lambda\alpha\}\cap\Omega_T\right|d\alpha+q\Lambda^{q}\int_{\alpha_0}^{\infty}\alpha^{q-1}\left|\{|\nabla u|>\Lambda\alpha\}\cap\Omega_T\right|d\alpha\\
			&=:I_1+I_2,
		\end{split}
	\end{equation}
	where we take \[\alpha_0=\lambda_0^{-1}=(\varepsilon\Lambda^{\sigma_1}/C_0)^{-\frac{1}{\sigma_1+2-p}}\|\nabla u\|_{L^{\sigma_1}(\Omega_T)}^{\frac{\sigma_1}{\sigma_1+2-p}}.\]
	It is easy to check that
	\begin{equation*}
		\begin{split}
			&I_1\leq q\Lambda^{q}\alpha_0^{q-p}\int_{0}^{\infty}\alpha^{p-1}\left|\{|\nabla u|>\Lambda\alpha\}\cap\Omega_T\right|d\alpha\leq q\Lambda^{q-p}\alpha_0^{q-p}\int_{\Omega_T}|\nabla u|^p\\
			&\quad\leq C\varepsilon^{{-\frac{q-p}{\sigma_1+2-p}}}\|\nabla u\|_{L^{\sigma_1}(\Omega_T)}^{\frac{\sigma_1(q-p)}{\sigma_1+2-p}}\int_{\Omega_T}|\nabla u|^p.
		\end{split}
	\end{equation*}
	Note that $p<\sigma_1<q$, by interpolation inequality and Young's inequality we obtain	
	\begin{equation}\label{I1}
		\begin{split}	I_1&\leq C(\varepsilon)\left(\int_{\Omega_T}|\nabla u|^p\right)^{\frac{2-p+q}{2}}+\varepsilon\int_{\Omega_T}|\nabla u|^q\\
			&\overset{\eqref{DuLp}}
			\leq C(\varepsilon)\left(\int_{\Omega_T}|F|^p\right)^{\frac{2-p+q}{2}}+\varepsilon\int_{\Omega_T}|\nabla u|^q.
			% 	&\leq C(\varepsilon)|\Omega_T|^\frac{(q-p)(2-p+q)}{2q}\left(\int_{\Omega_T}|F|^q\right)^{\frac{p(2-p+q)}{2q}}+\varepsilon\int_{\Omega_T}|\nabla u|^q,
		\end{split}
	\end{equation}

	Next we deal with $I_2$. First, we estimate \[\left|\{|\nabla u|>\Lambda\alpha\}\cap\Omega_T\right|.\]
	Note that $|\nabla u|(x,t)\leq \mathbf{M}^{\lambda}(|\nabla u|)(x,t)$
	for any $\lambda>0$ and $(x,t)\in \Omega_T$. Let $\lambda=\frac{1}{\alpha}$, we can use Lemma \ref{cover} and \eqref{M} to get
	\begin{equation*}
		\begin{split}
			\left|\{|\nabla u|>\Lambda\alpha\}\cap\Omega_T\right|&\leq \left|\{\mathbf{M}^\lambda(|\nabla u|^{\sigma_1})>(\Lambda\alpha)^{\sigma_1}\}\cap\Omega_T\right|\\&\leq C\varepsilon \left|\{{\bf M}^\lambda( |\nabla u|^{\sigma_1}) >\alpha^{\sigma_1}\}\cap\Omega_T\right|+ \left|\{\mathbf{M}^\lambda
			(|F|^{\sigma_1})>(\delta\alpha)^{\sigma_1}\}\cap\Omega_T\right|\\&\qquad+\left|\{{\bf M}^\lambda (|\nabla u|^{\sigma_1+\kappa_0}) >(\alpha/\varepsilon)^{\sigma_1+\kappa_0}\}\cap\Omega_T\right|\\
			&\leq \frac{C\varepsilon}{\alpha^{\sigma_1}}\int_{\{|\nabla u|>\alpha/2\}\cap\Omega_T}|\nabla u|^{\sigma_1}+\frac{C}{(\delta\alpha)^{\sigma_1}}\int_{\{|F|>\delta\alpha/2\}\cap\Omega_T}|F|^{\sigma_1}\\
			&\qquad+\frac{C\varepsilon^{\sigma_1+\kappa_0}}{\alpha^{\sigma_1+\kappa_0}}\int_{\{|\nabla u|>\alpha/(2\varepsilon)\}\cap\Omega_T}|\nabla u|^{\sigma_1+\kappa_0}.
		\end{split}
	\end{equation*}
	Hence,
	\begin{align*}
		I_2&\leq C\varepsilon\int_{\alpha_0}^{\infty}\alpha^{q-\sigma_1-1}\left[\int_{\{|\nabla u|>\alpha/2\}\cap\Omega_T}|\nabla u|^{\sigma_1}\right]d\alpha+\frac{C}{\delta^{\sigma_1}}\int_{\alpha_0}^{\infty}\alpha^{q-\sigma_1-1}\left[\int_{\{|F|>\alpha\delta/2\}\cap\Omega_T}|F|^{\sigma_1}\right]d\alpha\\&\qquad\qquad+C\varepsilon^{\sigma_1+\kappa_0}\int_{\alpha_0}^{\infty}\alpha^{q-(\sigma_1+\kappa_0)-1}\left[\int_{\{|\nabla u|>\alpha/(2\varepsilon)\}\cap\Omega_T}|\nabla u|^{\sigma_1+\kappa_0}\right]d\alpha.
	\end{align*}
	By Fubini's theorem, we have
	\begin{align*}
		&\int_{\alpha_0}^{\infty}\alpha^{q-\sigma_1-1}\left[\int_{\{|\nabla u|>\alpha/2\}\cap\Omega_T}|\nabla u|^{\sigma_1}\right]d\alpha\leq  \int_{\Omega_T}\left[\int_{0}^{2|\nabla u|}\alpha^{q-\sigma_1-1}d\alpha\right]|\nabla u|^{\sigma_1}\leq  C \int_{\Omega_T} |\nabla u|^{q}.
	\end{align*}
	Similarly, we have
	\begin{align*}
		&\int_{\alpha_0}^{\infty}\alpha^{q-\sigma_1-1}\left[\int_{\{|F|>\alpha\delta/2\}\cap\Omega_T}|F|^{\sigma_1}\right]d\alpha\leq \frac{C}{\delta^{q-\sigma_1}}\int_{\Omega_T}|F|^{q},
	\end{align*}
	and
	\begin{align*}
		&\int_{\alpha_0}^{\infty}\alpha^{q-(\sigma_1+\kappa_0)-1}\left[\int_{\{|\nabla u|>\alpha/(2\varepsilon)\}\cap\Omega_T}|\nabla u|^{\sigma_1+\kappa_0}\right]d\alpha\leq  C \varepsilon^{q-(\sigma_1+\kappa_0)}\int_{\Omega_T} |\nabla u|^{q}.
	\end{align*}
	Thus,
	\begin{align}\label{I2}
		I_2\leq \tilde{C}\varepsilon\int_{\Omega_T} |\nabla u|^{q}+\frac{\tilde{C}}{\delta^{q}}\int_{\Omega_T}|F|^{q}.
	\end{align}
	Plugging \eqref{I1} and \eqref{I2} into \eqref{I1I2} we obtain
	\begin{align*}
		\int_{\Omega_T}|\nabla u|^{q}\leq (\tilde{C}+1)\varepsilon \int_{\Omega_T}|\nabla u|^{q}+\frac{C}{\delta^{q}}\int_{\Omega_T}|F|^{q}+C(\varepsilon)\left(\int_{\Omega_T}|F|^p\right)^{\frac{2-p+q}{2}}.
	\end{align*}
	Let us take $\varepsilon$ small enough to get $(\tilde{C}+1)\varepsilon<\frac{1}{2}$, then we can determine the corresponding $\delta$ and $C(\varepsilon)$. Therefore
	\begin{align*}
		\int_{\Omega_T}|\nabla u|^{q}\leq C\int_{\Omega_T}|F|^{q}+C\left(\int_{\Omega_T}|F|^p\right)^{\frac{2-p+q}{2}},
	\end{align*}
	provided $\nabla u \in L^{\sigma_1+\kappa_0}(\Omega_T)$.\\
	\emph{Step 2}. To complete the proof, we consider the approximate system. For any $0<\kappa<\kappa_1$, let  $u_\kappa$ be a weak solution to 
	\begin{equation*}
		\left\{
		\begin{array}
			[c]{ll}%
			\partial_tu_\kappa -\operatorname{div}(a_\kappa(x,t)|\nabla u_\kappa|^{p-2}\nabla u_\kappa) = \operatorname{div} (|F_\kappa|^{p-2}F_\kappa)&\text{in }\Omega_T,\\
			u_\kappa=0&\text{on}\
			\partial_{p} \Omega_T.
			\\
		\end{array}
		\right.
	\end{equation*}
	where $a_\kappa(x,t)=a(x,t)\star\varphi_\kappa(x)$, $F_\kappa=F(x,t)\star\varphi_\kappa(x)$ and $\varphi_\kappa(x)$ is the standard mollifier, i.e., $\varphi_\kappa(x)=\kappa^{-n}$$\varphi (\frac{x}{\kappa})$ where $\varphi\in C_c^\infty(B_1(0))$ is a nonnegative radial function with $\|\varphi\|_{L^1(\mathbb{R}^n)}=1$. Clearly, $F_\kappa \in L^q(0,T; C^\infty (\Omega))$. Moreover,  $\|F_\kappa\|_{L^m(\Omega_T)}\leq \|F\|_{L^m(\Omega_T)}$ for $m\in \{p,q\}$ and we can take $\kappa_1$ small enough such that $[a_\kappa]_{R_0}\leq 3[a]_{R_0}$. By the classical regularity theory, we have $\nabla u_\kappa\in L^\infty(\Omega_T)$. Apply the result in Step 1 with $[a]_{R_0}<\delta_1/3=:\delta_0$, we get 
	\begin{equation}\label{kap}
		\begin{split}
			\int_{\Omega_T}|\nabla u_\kappa|^{q}&\leq C\int_{\Omega_T}|F_\kappa|^{q}+C\left(\int_{\Omega_T}|F_\kappa|^p\right)^{\frac{2-p+q}{2}}\\& \leq C\int_{\Omega_T}|F|^{q}+C\left(\int_{\Omega_T}|F|^p\right)^{\frac{2-p+q}{2}}.
		\end{split}
	\end{equation}
	By uniqueness of the problem \eqref{e1} in $C([0,T], L^2(\Omega))\cap L^p(0,T; W^{1,p}(\Omega))$, we  have 
	\begin{equation*}
		\nabla u_\kappa \to \nabla u
	\end{equation*}
	strongly in $L^p(\Omega_T)$ as $\kappa \to 0$. Hence there exists a subsequence, which can be still denoted by $\{u_\kappa\}$, such that $\nabla u_\kappa \to \nabla u$ a.e. as $\kappa\to 0$. Then, letting $\kappa\to 0$ in \eqref{kap} and using Fatou's lemma, one has 
	\begin{align*}
		\int_{\Omega_T}|\nabla u|^{q} \leq C\int_{\Omega_T}|F|^{q}+C\left(\int_{\Omega_T}|F|^p\right)^{\frac{2-p+q}{2}}.
	\end{align*}
	We complete the proof.
\end{proof}
\begin{remark}\label{repl}
	If $p>\frac{2n}{n+2}$, we can follow the above proof with $\sigma_1=p$ and $\kappa_0=0$, which finally leads to $$
		\int_{\Omega_T}|\nabla u|^{q}\leq C\int_{\Omega_T}|F|^{q}+C\left(\int_{\Omega_T}|F|^p\right)^{\frac{2-p+q}{2}},
	$$
	for any $q\geq p$.
\end{remark}
\textbf{Acknowledgments:} Quoc-Hung Nguyen is  supported by  Academy of Mathematics and Systems Science, Chinese Academy of Sciences startup fund and the National Natural Science Foundation of China (12050410257).
Na Zhao is supported by the Shanghai University of Finance and Economics startup fund.

\end{document}